\documentclass[a4paper]{amsart}

\usepackage{amsmath}
\usepackage{amssymb}
\usepackage{amsthm}
\usepackage{array}
\usepackage{subfigure}
\usepackage[all]{xy}
\usepackage{graphics,graphicx}
\usepackage{enumitem}
\usepackage{multicol}
\usepackage{mathtools}
\usepackage{url}
\label{eq-ODE}

\usepackage[short,nodayofweek]{datetime}
 \usepackage[bookmarks=false]{hyperref}
\usepackage{nicefrac}

\usepackage[font={small}]{caption}

\usepackage{longtable}
\usepackage{enumitem}

\usepackage{pdflscape}
\usepackage{afterpage}
\usepackage{capt-of}

\usepackage{tikz,pgfpages}
\usetikzlibrary{matrix,arrows}
\usetikzlibrary{arrows,arrows.meta,shapes,positioning,calc,patterns,cd,decorations.pathmorphing}

\newcommand{\Depth}{2}
\newcommand{\Height}{2}
\newcommand{\Width}{2}

\usepackage{fancybox}
\usepackage{todonotes}
\usepackage{menukeys}
\newcounter{todocounter}

\presetkeys{todonotes}{inline, color=blue!30}{}

\DeclareFontFamily{OT1}{rsfs}{}
\DeclareFontShape{OT1}{rsfs}{n}{it}{<-> rsfs10}{}
\DeclareMathAlphabet{\mathscr}{OT1}{rsfs}{n}{it}

\newcommand{\RR}{\mathbb R}

\newcommand{\CC}{\mathbb C}

\newcommand{\kk}{\Bbbk}

\newcommand{\sth}{\;\middle|\;}

\newtheorem{thm}{Theorem}[section]
\newtheorem*{thm*}{Theorem}

\newtheorem*{cor*}{Corollary}
\newtheorem{prop}[thm]{Proposition}

\theoremstyle{definition}
\newtheorem{defn}[thm]{Definition}

\theoremstyle{remark}
\newtheorem{rmk}[thm]{Remark}
\newtheorem{eg}[thm]{Example}

\newcommand{\done}{\hfill $\triangleleft$}

\title{The geometry of sloppiness}

\author[Dufresne]{Emilie Dufresne}
\address{School of Mathematical Sciences, University of Nottingham, University Park, Nottingham, NG7 2RD}
\email{emilie.dufresne@nottingham.ac.uk}

\author[Harrington]{Heather A Harrington}
\address{Mathematical Institute, University of Oxford, Andrew Wiles Building, Radcliffe Observatory Quarter, Woodstock Road, Oxford, OX2 6GG}
\email{harrington@maths.ox.ac.uk}

\author[Raman]{Dhruva V Raman}
\address{Engineering Department, University of Cambridge, Trumpington Street, Cambridge, CB2 1PZ}
\email{dhruva.raman@eng.cam.ac.uk }

\begin{document}

\date{\today}
%


\begin{abstract}
The use of mathematical models in the sciences often involves the estimation of unknown parameter values from data. Sloppiness provides information about the uncertainty of this task. In this paper, we develop a precise mathematical foundation for sloppiness and define rigorously its key concepts, such as `model manifold', in relation to concepts of structural identifiability. We redefine sloppiness conceptually as a comparison between the premetric on parameter space induced by measurement noise  and a reference metric. This opens up the possibility of alternative quantification of sloppiness, beyond the standard use of the Fisher Information Matrix, which assumes that parameter space is equipped with the usual Euclidean metric and the measurement error is infinitesimal. Applications include parametric statistical models, explicit time dependent models, and ordinary differential equation models.
\end{abstract}

\maketitle


\section{Introduction}\label{Section-Introduction}

Mathematical models describing physical, biological, and other real-life phenomena contain parameters whose values must be estimated from data.  Over the past decade, a powerful framework called ``sloppiness'' has been developed that relies on Information Geometry \cite{Amari2007} to study the uncertainty in this procedure \cite{Brown2003,Daniels2008,Transtrum:2010ci,Transtrum:2011de,Transtrum:2014hr,Transtrum:2015hm}. Although the idea of using the Fisher Information to quantify uncertainty is not new (see for example \cite{Fisher1925, rao1945information}), the study of sloppiness gives rise to a particular observation about the uncertainty of the procedure and has potential implications beyond parameter estimation. Specifically, sloppiness has enabled advances in the field of systems biology, drawing connections to sensitivity \cite{gutenkunst:pcb:2007,Erguler:2011bu,sloppycell}, experimental design \cite{Apgar:2010di, Mannakee2016,gutenkunst:pcb:2007}, identifiability \cite{ar-jk-mps-mj-jt:comparisonIdentifiability,Transtrum:2015hm,Chis2016:SloppinessIdentifiability}, robustness \cite{Daniels2008}, and reverse engineering \cite{Erguler:2011bu,Clermont:2015}.  Sethna, Transtrum and co-authors identified sloppiness as a universal property of highly parameterised mathematical models \cite{Waterfall2006,Transtrum:2010ci,Tonsing2014,gutenkunst:pcb:2007}.   However, the precise interpretation of sloppiness remains a matter of active discussion in the literature \cite{Apgar:2010di,keh-trm-rwa:nonlinIdent,mf-asm-ak:bootstrap}. 

This paper's main contribution is to provide a unified mathematical framework for sloppiness rooted in algebra and geometry. We extend the concept beyond time dependent models, in particular, to statistical models. We rigorously define the concepts and building blocks for the theory of sloppiness. Our approach requires techniques from many fields including algebra, geometry, and statistics.
We illustrate each new concept with a simple concrete example. The new mathematical foundation we provide for sloppiness is not limited by current computational tools and opens up the way to further work.

Our general setup is a mathematical model $M$ that describes the behavior of a variable $x\in\RR^m$ depending on a parameter $p\in P\subseteq \RR^r$. Our first step is to explain how each precise choice of perfect data $z$ induces an equivalence relation $\sim\!\!_{M,z}$ on the parameter space:  two parameters are equivalent if they produce the same perfect data. We then characterize the various concepts of structural identifiability in terms of the equivalence relation $\sim\!\!_{M,z}$. Roughly speaking, structural identifiability asks to what extent perfect data determines the value of the parameters. See section \ref{Section-EquivalenceRelation}.

Assume that the perfect data $z$ is a point of $\RR^N$ for some $N$. The second crucial step needed in order to define sloppiness is a map $\phi$  from parameter space $P$ to data space $\RR^N$ giving the perfect data as a function $\phi(p)$ of the parameters known as a ``model manifold" in the literature \cite{Transtrum:2010ci,Transtrum:2011de,Transtrum:2014hr,Transtrum:2015hm}, which we rename as a model prediction map.  A model prediction map thus induces an injective function on the set of equivalence classes (the set-theoretic quotient $P/{\sim\!\!_{M,z}}$), that is, the equivalence classes can be separated by $N$ functions $P\to \RR$. See Section \ref{Section-ModelPredictions}.

The next step is to assume that the mathematical model describes the phenomenon we are studying perfectly, but that the ``real data'' is corrupted by measurement error and the use of finite sample size. That is, we assume that noisy data arises from a random process whose probability distribution then induces a premetric $d$ on the parameter space, via the Kullback-Leibler divergence ( see start of Section \ref{Section-Sloppiness} ). This premetric $d$ quantifies the proximity between the two parameters in parameter space via the discrepancy between the probability distributions of the noisy data associated to the two parameters. 

The aforementioned premetric $d$ has a tractable approximation in the limit of decreasing measurement noise using the Fisher Information Matrix (FIM). In the standard definition, a model is ``sloppy'' when the condition number of the FIM is large, that is, there are several orders of magnitude between its largest and smallest eigenvalues. Multiscale sloppiness (see \cite{dvr-ja-ap:SloppinessDhruva}) extends this concept to regimes of non-infinitesimal noise. 

We conceptually extend the notion of sloppiness to a comparison between the premetric $d$ and a reference metric on parameter space. We demonstrate that using the condition number of the FIM to measure sloppiness at a parameter $p_0$, as is done in most of the sloppiness literature \cite{Brown2003,Daniels2008,Transtrum:2010ci,Transtrum:2011de,Transtrum:2014hr,Transtrum:2015hm}, corresponds to comparing an approximation of $d$ in an infinitesimal neighborhood of $p_0$ to the standard Euclidean metric on $\RR^r\supset P$. Note that considering the entire spectrum of the FIM, as is done newer work in the sloppiness literature (eg, \cite{Waterfall2006}) corresponds to performing a more refined comparison between an approximation of $d$ in an infinitesimal neighborhood of $p_0$ to the standard Euclidean metric on $\RR^r\supset P$. Multiscale sloppiness, which we extend here beyond its original definition \cite{dvr-ja-ap:SloppinessDhruva}  for Euclidean parameter space and Gaussian measurement noise, avoids approximating $d$, and so better reflects the sloppiness of models beyond the infinitesimal scale. Finally, we describe the intimate relationship between sloppiness and practical identifiability, that is, whether generic noisy data allows for bounded confidence regions when performing maximum likelihood estimation. See Section \ref{Section-Sloppiness}.

The following diagram illustrates the main objects discussed in this paper:\\

\includegraphics[width=\textwidth]{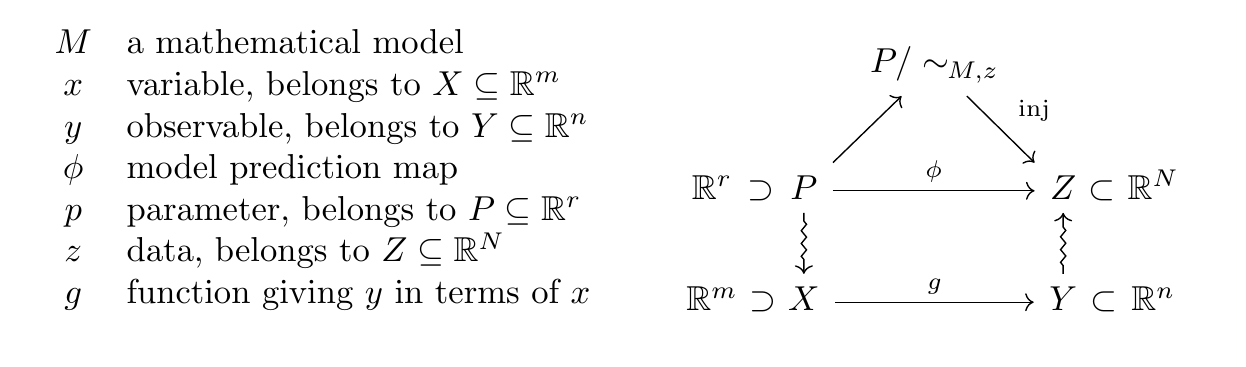}


\section{An equivalence relation on parameter space}\label{Section-EquivalenceRelation}

A mathematical model $M$ describes the behavior of a variable $x\in X\subseteq \RR^m$ depending on a parameter $p\in P\subset \RR^r$,  with measurable output $y=g(x)\in Y\in \RR^n$. We further specify a  choice of perfect data $z$ produced for the parameter value $p$. The nature of perfect data will be made clear in the examples discussed throughout the section. We think of the perfect data $z$ as belonging to the wider \emph{data space} $Z$ that encompasses all possible ``real'' data. Data space will be defined rigorously  in Section \ref{Section-Sloppiness} when measurement noise comes into play.

An example where the measurable output $y$ differs from $x=(x_1,\ldots,x_n)$ is when only some of the $x_i$'s can be measured (e.g., due to cost or inaccessibility of certain variables). The perfect data is extracted from the measurable output, as illustrated by examples \ref{eg-2BiasedCoins1}, \ref{eg-Line1}, and \ref{eg-ODESoln1}.
The behavior of the variable $x$ may also vary in time (and position in space, although this will not be addressed here). In the time dependent case, the perfect data often consists of values of the measurable output $y$ at finitely many timepoints, that is, a \emph{time series}.  An alternative choice of perfect data would be the set of all stable steady states. We are also interested in what we will call the \emph{continuous data}, that is, the value of $y$ at all possible timepoints or, equivalently, the function $t\mapsto y(t)$ for $t$ belonging to the full time interval. For a statistical model, the measurable output is the outcome from one instance of a statistical experiment, while a natural choice for perfect data is a probability distribution belonging to the model, or any function or set of functions characterising this probability distribution. 

Given a model $M$, a choice of perfect data $z$ induces a \emph{model-data equivalence relation $\sim\!\!_{M,z}$} on the parameter space $P$ as follows: two parameters $p$ and $p'$ are equivalent ($p\sim\!\!_{M,z}p'$) if and only if fixing the parameter value to $p$ or $p'$  produces the same perfect data. We now provide a more concrete description for a selection of types of mathematical models.


\subsection{Finite discrete statistical models}

The most straightforward case is when the perfect data is described explicitly as a function of the parameter $p$. Finite discrete statistical models fall within this group, with the perfect data $z$ being the probability distribution of the possible outcomes depending on the choice of parameter.  Such a model is described by a map
\vspace{-0.2cm}
\begin{align*}
\rho\colon P &\to  [0,1]^n\\
                     p &\mapsto (\rho_1(p),\ldots,\rho_n(p)).
\end{align*}
The model-data equivalence relation then coincides with the equivalence relation $\sim\!\!_\rho$ induced on $P$ by the map $\rho$, that is, $p\sim\!\!_{M,z} p' \text{ if and only if } \rho(p)=\rho(p')$.

\begin{eg}[Two biased coins \cite{sh-ak-bs:sle}]\label{eg-2BiasedCoins1}
A person with two biased coins, picks one at random, tosses it and records the result. The person then repeats this three additional times, for a total of four coin tosses. The parameter is $(p_1,p_2,p_3)\in[0,1]^3$, where $p_1$ is the probability of picking the first coin, $p_2$ is the probability of obtaining heads when tossing the first coin (that is, the bias of the first coin), and $p_3$ is the probability of obtaining heads when tossing the second coin. Here, the measurable output is the record of a single instance of the statistical experiment described and perfect data is the probability distribution of the possible outcomes (there are five possibilities). The map giving the model is then
\vspace{-0.2cm}
\begin{alignat*}{3}
\rho &\colon & [0,1]^3 &\to&& \RR^5\\
                & & (p_1,p_2,p_3) &\mapsto&& (\rho_0, \rho_1, \rho_2, \rho_3, \rho_4),\end{alignat*}                 
where $\rho_i$ is the probability of obtaining heads $i$ times. Explicitly we have
\begin{align*}
\rho_0 &= p_1(1-p_2)^4+(1-p_1)(1-p_3)^4,\\
\rho_1 &=  4p_1p_2(1-p_2)^3+4(1-p_1)p_3(1-p_3)^3,\\
\rho_2 &=  6p_1p_2^2(1-p_2)^2+6(1-p_1)p_3^2(1-p_3)^2,\\
\rho_3 &=  4p_1p_2^3(1-p_2)+4(1-p_1)p_3^3(1-p_3),\\
\rho_4 &=  p_1p_2^4+(1-p_1)p_3^4.
\end{align*}
Two parameters $(p_1,p_2,p_3)$ and $(p'_1,p'_2,p'_3)$ are then  equivalent if  $\rho(p_1,p_2,p_3)=\rho(p'_1,p'_2,p'_3)$, or equivalently, if $\rho_i(p_1,p_2,p_3)=\rho_i(p'_1,p'_2,p'_3)$ for each $i$.

We next study the equivalence classes. As we cannot distinguish between the two coins, we will always have $(p_1,p_2,p_3)\sim\!\!_{M,z} (1-p_1,p_3,p_2)$, and so the equivalence class of $(p_1,p_2,p_3)$ contains the set $\{(p_1,p_2, p_3),(1-p_1,p_3,p_2)\}$.
 Furthermore, the equivalence class of $(p_1,p_2,p_2)$ will contain $\{(q_1,p_2,p_2) \mid q_1\in [0,1]\}$. The equivalence class of $(0,p_2,p_3)$ will contain $\{(0,q_1,p_3) \mid q_1\in [0,1]\}$ and $\{(1,p_2,q_2) \mid q_2\in [0,1]\}$.

The ideal $(\rho_i\otimes 1-1\otimes \rho_i \mid i=0,\ldots ,4)$ in $\CC[p_1,p_2,p_3]\otimes \CC[p_1,p_2,p_3]$ is the ideal cutting out the set-theoretic equivalence relation $\sim\!\!_\rho$ on $\CC^3$ induced by extending the function $\rho$ to $\CC^3$. Indeed, the zero set of this ideal is the set of pairs $((p_1,p_2,p_3),(p_1',p_2',p_3'))\in \CC^3\times \CC^3$ such that  $(p_1,p_2,p_3)\sim\!\!_\rho (p_1',p_2',p_3')$. Using a symbolic computation software, we compute the prime decomposition of its radical and conclude that the equivalence class of $(p_1,p_2,p_3)\in\CC^3$ is
\begin{align*}
&\{(p_1,p_2,p_3),(1-p_1,p_3,p_2)\}     & \text{ if }p_1\neq 0,1,1/2~p_2\neq p_3,\\
&\{(q,p_2,p_2) \mid q\in \CC\}              & \text{ if } p_1\neq 0,1,1/2~p_2=p_3,\\
&\{(0, q_1, p_3) \mid q_1\in\CC\}\cup \{(1,p_2,q_2) \mid q_2\in\CC\} & \text{ if } p_1=0,1,\\
&\{(1/2,p_2,p_3)\}                                                     & \text{ if }p_1=1/2.
\end{align*}
Therefore, the equivalence classes in $[0,1]^3$ must be contained in the intersections of the above sets with $[0,1]^3$. Thus the equivalence class of $(p_1,p_2,p_3)\in [0,1]^3$ is
\begin{align*}
&\{(p_1,p_2,p_3),(1-p_1,p_3,p_2)\}     & \text{ if }p_1\neq 0,1,1/2~p_2\neq p_3,\\
&\{(q,p_2,p_2) \mid q\in [0,1]\}              & \text{ if } p_1\neq 0,1,1/2~p_2=p_3,\\
&\{(0, q, p_3) \mid q_1\in [0,1] \}\cup \{(1,p_2,q) \mid q_2\in [0,1]\} & \text{ if } p_1=0,1,\\
&\{(1/2,p_2,p_3)\}                                                     & \text{ if }p_1=1/2.
\end{align*}
In particular, we obtain a stratification of parameter space as shown in Fig.~\ref{fig:coin}. 

\begin{figure}[ht]
\centering
\begin{tikzpicture}
\coordinate (O) at (0,0,0);
\coordinate (A) at (0,\Width,0);
\coordinate (B) at (0,\Width,\Height);
\coordinate (C) at (0,0,\Height);
\coordinate (D) at (\Depth,0,0);
\coordinate (E) at (\Depth,\Width,0);
\coordinate (F) at (\Depth,\Width,\Height);
\coordinate (G) at (\Depth,0,\Height);
\coordinate (H) at (\Depth /2,-.2,\Height);
\coordinate (I) at (2.6,.20,\Height);
\coordinate (J) at (3,1.5 ,\Height);
\coordinate (K) at (\Depth,\Width,\Height);
\coordinate (L) at (\Depth/2,0,0);
\coordinate (M) at (\Depth/2,\Width,0);
\coordinate (N) at (\Depth/2,\Width,\Height);
\coordinate (P) at (\Depth /2,0,\Height);

\draw[gray,fill=white!30] (O) -- (C) -- (G) -- (D) -- cycle;
\draw[gray,fill=white!30] (O) -- (A) -- (E) -- (D) -- cycle;
\draw[blue,fill=blue!30,opacity=0.6] (C) -- (G) -- (E) -- (A) -- cycle;
\draw[green,fill=green!30,opacity=0.6] (O) -- (A) -- (B) -- (C) -- cycle ;
\draw[green,fill=green!30,opacity=0.6] (D) -- (E) -- (F) -- (G) -- cycle;
\draw[gray,fill=gray!30,opacity=0.6] (L) -- (M) -- (N) -- (P) -- cycle;
\draw[blue,fill=blue!30,opacity=0.2] (P) -- (M);

\draw[gray,fill=white!20,opacity=0.2] (A) -- (B) -- (F) -- (E) -- cycle;
\draw (H) node {$p_1$};
\draw (I) node {$p_2$};
\draw (J) node {$p_3$};
\draw (C) node [label={[xshift=-0.1cm, yshift=-0.5cm]0}]{};

\end{tikzpicture}
\caption{Stratification of parameter space for the two biased coins example. Blue: $\{(p_1,p_2,p_3)\mid p_1\neq 0,1,1/2~p_2=p_3\}$ Green: $\{(p_1,p_2,p_3)\mid p_1=0,1,1/2\}$, Grey: $\{(p_1,p_2,p_3)\mid p_1=1/2\}$ the rest of the cube (interior and faces) is the generic part $\{(p_1,p_2,p_3)\mid p_1\neq 0,1,~p_2\neq p_3\}$.}
\label{fig:coin}
\end{figure}
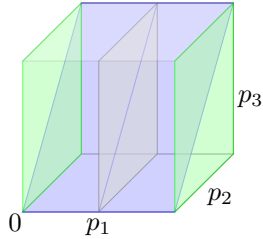
We remark that almost all equivalence classes have dimension zero, although some equivalence classes have dimension one. As the points with zero-dimensional equivalence classes form a dense open subset of parameter space, we say that the dimension of an equivalence class is generically zero. Note that  since all these zero-dimensional equivalence classes have size two, we say that the equivalence classes are generically of size two. \done\end{eg}


\subsection{time dependent models and the $2r+1$ result}

Let $M$ be an explicit time dependent model with measurable output $x$. That is, the behavior of the variable $x$ is given by the map
\begin{alignat*}{3} 
      \rho &\colon & P\times \RR_{\geq 0} &\to&& \RR^m  \\ 
                        && (p,t) &\mapsto && x(p,t),
\end{alignat*} 
and $x$ can be measured at any time $t$. Perfect time series data produced by the parameter $p$ will be $(x(p,t_1),\ldots,x(p,t_N))$, where $0\leq t_1<\cdots<t_N\in\RR_{\geq0}$ are timepoints. We denote the corresponding model-data equivalence relation on $P$ by $\sim\!\!_{M,t_1,\ldots,t_N}$. The continuous data is the map $\RR_{\geq 0}\to \RR^m$ given by $t\mapsto x(p,t)$. We denote the equivalence relation induced  by the continuous data on $P$ by $\sim_{M,\infty}$.

We particularly consider ODE systems with time series data. For such a model $M$, the behavior of the variable $x$ is described by a system of ordinary differential equations depending on the parameter $p\in P$ with some initial conditions:
\begin{alignat}{3}\label{eq-ODE1}
&\dot x &=& f(p,x)\\
&x(0)&=& x_0.\nonumber
\end{alignat}
When initial conditions are known or we do not wish to estimate them, they are not considered as components of the parameter. The measurable output is $y=g(x)$, and perfect data is then $(y(t_1),\ldots,y(t_N))\in\RR^{Nn}$ for $0\leq t_1<\cdots<t_N\in\RR_{\geq0}$. The continuous data is given by the function $\RR_{\geq0}\to\RR^n,~t\mapsto y(t)$, which supposes that a solution to the given ODE system exists, a valid assumption in the real-analytic case. 

The key result when working with time dependent models with time series data is the $2r+1$ result of Sontag \cite[Theorem 1]{Sontag:2rplus1}, which implies that there is a single ``global'' model-data equivalence relation: the equivalence relation $\sim_{M,\infty}$ induced by the continuous data. Precisely, we suppose that the model $M$ is real-analytic, that is, either an explicit time-dependent model given by a real-analytic map or an ODE system as in (\ref{eq-ODE1}) with $f$ a real-analytic function. We additionally assume that the variable $x$, the parameter $p$, and the time variable $t$  belong to real-analytic manifolds. If we suppose that $P$ is a real-analytic manifold of dimension $r$, then for $N\geq 2r+1$ and a generic choice of timepoints $t_1,\ldots, t_N$ the equivalence relation $\sim\!\!_{M,t_1,\ldots,t_N}$ coincides with the equivalence relation $\sim_{M,\infty}$. 

An important consequence of the $2r+1$ result  \cite{Sontag:2rplus1} is that for real-analytic time-dependent models with time series data, the model equivalence relation is a \emph{global} structural property of the model, and one need not specify which exact timepoints are used.

\begin{rmk}
Note that in many applications the variable $x$ belongs to the real positive orthant, which is indeed a real-analytic manifold. The condition on the time variable can be relaxed to include closed and partially closed time intervals. 
\end{rmk}

\begin{rmk}
A choice of $N$ timepoints corresponds to a choice of a point in the real analytic manifold $T:=\{(t_1,\ldots,t_N)\in \RR_{\geq 0} \mid t_i<t_{i+1}\}$. The use of the word ``generic'' in the statement means that there can be choices of $N$ timepoints that will not induce the equivalence relation $\sim\!\!_{M,\infty}$, but that these choices of timepoints will belong to a small subset of $T$, so small that its complement contains an open dense subset of $T$.
\end{rmk}

In cases where no results like the  $2r+1$ result \cite{Sontag:2rplus1}  hold, there is no ``global'' equivalence relation. Therefore, a finite number of measurements will never induce the same equivalence relation on parameter space as the continuous data. In other words, by taking more and more measurements we could obtain an increasingly fine equivalence relation without ever converging to $\sim\!\!_{M,\infty}$.

\begin{eg}[{A model for which the $2r+1$ result does  not hold, cf \cite[Section 2.3]{Sontag:2rplus1}}]
The model, while artificial, is an explicit time dependent model given by the map:
\begin{alignat*}{3}
\rho&\colon & \RR_{>0}\times\RR_{\geq 0} &\to&& \RR\\
                      && (p,t) &\mapsto&& \gamma(p-t),
\end{alignat*}
where $\gamma\colon\RR\to\RR$ is a $C^\infty$ map that is $e^{1/s}$ for $s<0$ and zero for $s\geq 0$. 
Suppose for a contradiction that evaluating at timepoints $t_1,\ldots,t_N$ induces the same equivalence relation on $\RR_{>0}$ as taking the perfect data to be the maps $t\mapsto \rho(p,t)$. Take $p_1>p_2\geq t_N$, it follows that $\rho(p_1,t_i)=0=\rho(p_2,t_i)$ for each $i=1,\ldots,N$. On the other hand, we will have $\rho(p_1,\nicefrac{p_1+p_2}{2})=0\neq\rho(p_2,\nicefrac{p_1+p_2}{2})$, and so we have a contradiction.\done\end{eg}

\begin{eg}[Fitting points to a line]\label{eg-Line1} This example is motivated by one of the examples found on the webpage of Sethna dedicated to sloppiness  \cite{js:FittingPolynomials}. We consider an explicit time dependent model where the variable $x$ changes linearly in time:
\[x(t)=a_0+a_1t,\]
that is, $x$ is given as a polynomial function in $t$ depending on the parameter $(a_0,a_1)\in \RR^2$. Hence by the $2r+1$ result \cite{Sontag:2rplus1}, taking the perfect data to be the measurement at $2\cdot 2+1=5$ sufficiently general time points induces the same equivalence relation as taking the perfect data as the continuous function $t\mapsto a_0+a_1t$. In fact, taking measurements at two timepoints will suffice, since there is exactly one line going through any two given points.

We have that $(a_0,a_1)\sim_{M,\infty} (b_0,b_1)$ if and only if
\[a_0+a_1t=b_0+b_1t,~\text{for all } t\in\RR_{\geq0}.\]

It follows that $a_0=b_0$ (taking $t=0$), and then $a_1=b_1$ (taking $t=1$), thus $[(a_o,a_1)]_{M,\infty}=\{(a_0,a_1)\}$. Naturally, this coincides with the equivalence classes obtained with taking the perfect data to be noiseless measurements at $t=0$ and $t=1$, that is, $(x(0),x(1))=(a_0,a_0+a_1)$.\done\end{eg}

\begin{eg}[Sum of exponentials]\label{eg-SumExponentials1}
The sum of exponentials model for exponential decay, widely studied in the sloppiness literature \cite{Transtrum:2010ci,Transtrum:2011de,Transtrum:2014hr}, is an explicit time dependent model given by the function
\begin{alignat*}{3}
\rho&\colon & \RR_{\geq 0}^2 \times \RR_{\geq 0}& \to &&\RR \\
                     & & (a,b,t) &\mapsto&& e^{-at}+e^{-bt}.
\end{alignat*}

By the $2r+1$ result \cite{Sontag:2rplus1}, the time series $(e^{-at_1}+e^{-bt_1},\ldots,e^{-at_5}+e^{-bt_5})$ with $(t_1,\ldots,t_5)$ generic induces the same equivalence relation on the parameter space $\RR_{\geq 0}^2$ as the continuous data. This model is clearly non-identifiable. Indeed, for any $a,b\in \RR_{\geq 0}$, the parameters $(a,b)$ and $(b,a)$ yield the same continuous data since $e^{-at}+e^{-bt}=e^{-bt}+e^{-at}$ for all $t$. It follows that the equivalence class of a parameter $(a,b)$ will contain the set $\{(a,b),(b,a)\}$.

Suppose $(a,b)\sim_{M,t_1,t_2}(a',b')$ where $t_1\neq t_2$ are positive real numbers, thus
\begin{align*}
e^{-at_1}+e^{-bt_1} &= e^{-a't_1}+e^{-b't_1},\\
e^{-at_2}+e^{-bt_2} &= e^{-a't_2}+e^{-b't_2}.
\end{align*}
We can reduce it to the case $t_1=1,~t_2=2$ by rescaling the time variable via the substitution $t\mapsto \nicefrac{(t+t_2-2t_1)}{(t_2-t_1)}$ in $\rho$. Simplifying further with the substitution $x=e^{-a},y=e^{-b},u=e^{-a'},v=e^{-b'}$, the equation becomes:
\vspace{-0.2cm}
\begin{align*}
x+y&=u+v\\
x^2+y^2&=u^2+v^2.
\end{align*}
It is then easy to see that the only solutions $(u,v)$ to this system are $(u,v)=(x,y)$ or $(u,v)=(y,x)$. As the exponential function is injective it follows that $(a',b')=(a,b)$ or $(a',b')=(b,a)$.

Therefore, the equivalence class of a parameter $(a,b)$ is
\begin{align*}
&\{(a,b),(b,a)\}, & \text{ if }a\neq b,\\
&\{(a,a)\},           & \text{ if } a=b.\\
\end{align*}

\begin{figure}[ht]
\centering
\begin{tikzpicture}
\coordinate (O) at (0,0);
\coordinate (A) at (\Width,0);
\coordinate (Aa) at (2.2,0);
\coordinate (B) at (0,\Height);
\coordinate (Ba) at (0,2.2);
\coordinate (C) at (\Width,\Height);

\draw[white,fill=green!30,opacity=0.6] (O) -- (A) -- (C) -- (B) -- cycle;
\draw[green,fill=green!60] (O) -- (A) ;
\draw[gray,fill=green!60,->] (A) -- (Aa) ;
\draw[green,fill=green!60] (O) -- (B) ;
\draw[gray,fill=green!60,->] (B) -- (Ba) ;
\draw[blue,fill=blue!30,opacity=0.6] (O) -- (C) ;
\draw (Aa) node[below] {$a$};
\draw (Ba) node[left] {$b$};
\end{tikzpicture}
\caption{Parameter space of sum of exponential example. Green: $\{(a,b)\in\RR^2_{\geq0}\mid a\neq b\}$, Blue: $\{(a,a)\mid a\in\RR_{\geq0}\}$}
\label{fig:sum-exp}
\end{figure}
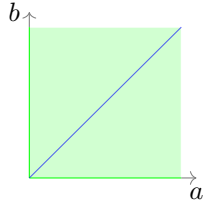
\done
\end{eg}

\begin{eg}[An ODE system with a solution]\label{eg-ODESoln1}
We consider the ODE system with variable $(x_1,x_2)\in\RR_{\geq0}^2$ and parameter $(p_1,p_2)\in\RR_{>0}^2$ given by
\begin{align*}
\dot{x}_1 &=-p_1x_1\\
\dot{x}_2 &=p_1x_1-p_2x_2
\end{align*}
with known initial conditions $x_1(0)=c_1$ and $x_2(0)=0$, and observable output $(x_1,x_2)$. Set $U:=\{(p_1,p_2)\mid p_1\neq p_2\}$. For $(p_1,p_2)\in U$, a solution to this system is given by
\vspace{-0.2cm}
\begin{align*}
x_1(t) &=c_1e^{-p_1t}\\
x_2(t) &=\frac{c_1p_1}{(p_2-p_1)}\left(e^{-p_1t}-e^{-p_2t}\right).
\end{align*}
When $p_1=p_2$, the ODE system becomes
\begin{align*}
\dot{x}_1 &=-p_1x_1\\
\dot{x}_2 &=p_1(x_1-x_2),
\end{align*}
and a solution is given by 
\vspace{-0.2cm}
\begin{align*}
x_1(t) &=c_1e^{-p_1t}\\
x_2(t) &={c_1p_1t}e^{-p_1t}.
\end{align*}

The $2r+1$ result \cite{Sontag:2rplus1} implies that, for  general $(t_1,\ldots,t_5)$, the time series data $(e^{-at_1}+e^{-bt_1},\ldots,e^{-at_5}+e^{-bt_5})$ induces the same equivalence relation on the parameter space $\RR_{\geq 0}^2$ as the continuous data. As in the previous example, we will show that this can be achieved by taking a time series with two distinct nonzero time points. We can again reduce to the case $t_1=1,t_2=2$. Suppose that $(p_1,p_2)$ and $(p_1',p_2')$ are two parameters that produce the same perfect data. The first case we consider is when they both belong to $U$, then we have
\begin{align*}
c_1e^{-p_1}&=c_1e^{-p_1'},\\
c_1e^{-2p_1}&=c_1e^{-2p_1'},\\
\frac{c_1p_1}{(p_2-p_1)}(e^{-p_1}-e^{-p_2})&=\frac{c_1p_1'}{(p_2'-p_1')}(e^{-p_1'}-e^{-p_2'}),\\
\frac{c_1p_1}{(p_2-p_1)}(e^{-2p_1}-e^{-2p_2})&=\frac{c_1p_1'}{(p_2'-p_1')}(e^{-2p_1'}-e^{-2p_2'}).
\end{align*}
The first equation implies that $p_1=p_1'$ since $c_1\neq0$ and the exponential function is injective. Using the last two equations we find that we have
\begin{align*}
e^{-p_1}+e^{-p_2}=\frac{\frac{c_1p_1}{(p_2-p_1)}(e^{-2p_1}-e^{-2p_2})}{\frac{c_1p_1}{(p_2-p_1)}(e^{-p_1}-e^{-p_2})}=\frac{\frac{c_1p_1'}{(p_2'-p_1')}(e^{-2p_1'}-e^{-2p_2'})}{\frac{c_1p_1'}{(p_2'-p_1')}(e^{-p_1'}-e^{-p_2'})}=e^{-p_1'}+e^{-p_2'},
\end{align*}
And since $p_1=p_1'$, it follows that $p_2=p_2'$. Next, if we suppose that neither belongs to $U$, that is, $(p_1,p_1)$ and $(p_1',p_1')$ produce the same perfect data, we then have
\begin{align*}
c_1e^{-p_1}&=c_1e^{-p_1'},\\
c_1e^{-2p_1}&=c_1e^{-2p_1'},\\
{c_1p_1}e^{-p_1}&={c_1p'_1}e^{-p'_1},\\
{2c_1p_1}e^{-2p_1}&={2c_1p'_1}e^{-2p'_1}.
\end{align*}
The first equation already implies that $p_1=p_1'$. Finally, we suppose that one parameter is in $U$ and the other is not, that is, $(p_1,p_2)$ with $p_1\neq p_2$ and $(p_1',p_1')$ produce the same perfect data. We then have
\vspace{-0.2cm}
\begin{align*}
c_1e^{-p_1}&=c_1e^{-p_1'}\\
c_1e^{-p_1^2}&=c_1e^{-p_1'^2}\\
\frac{c_1p_1}{(p_2-p_1)}(e^{-p_1}-e^{-p_2})&={c_1p'_1}e^{-p'_1}\\
\frac{c_1p_1}{(p_2-p_1)}(e^{-2p_1}-e^{-2p_2})&={2c_1p'_1}e^{-2p'_1}.
\end{align*}
The first two equations imply that $p_1=p'_1$ and so the last two equations become
\begin{align*}
\frac{c_1p_1}{(p_2-p_1)}(e^{-p_1}-e^{-p_2})     &={c_1p_1}e^{-p_1}\\
\frac{c_1p_1}{(p_2-p_1)}(e^{-2p_1}-e^{-2p_2}) &    ={2c_1p_1}e^{-2p_1}.
\end{align*}
If $p_1=0$, then $p_2$ is not further constrained. If $p_1\neq0$, the equations simplify to
\begin{align*}
\frac{1}{(p_2-p_1)}(e^{-p_1}-e^{-p_2})     &={}e^{-p_1}\\
\frac{1}{(p_2-p_1)}(e^{-2p_1}-e^{-2p_2}) &={2}e^{-2p_1},
\end{align*}
and so
\begin{align*}
e^{-p_1}+e^{-p_2}=\frac{\frac{1}{(p_2-p_1)}(e^{-2p_1}-e^{-2p_2})}{\frac{1}{(p_2-p_1)}(e^{-p_1}-e^{-p_2})}=\frac{{2}e^{-2p_1}}{e^{-p_1}}=2e^{-p_1}.
\end{align*}
But this implies that $p_1=p_2$, a contradiction. Hence, the third case was not possible in the first place.

We conclude that the equivalence class of the parameter $(p_1,p_2)\in P$ is
\begin{align*}
&\{(p_1,p_2)\}                                      & \text{ if } p_1\neq 0,\\
&\{(0,q) \mid q\in \RR_{\geq 0}\} & \text{ if }p_1=0.
\end{align*}\done
\end{eg}

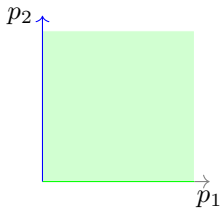
\begin{figure}
\centering
\begin{tikzpicture}
\coordinate (O) at (0,0);
\coordinate (A) at (\Width,0);
\coordinate (Aa) at (2.2,0);
\coordinate (B) at (0,\Height);
\coordinate (Ba) at (0,2.2);
\coordinate (C) at (\Width,\Height);

\draw[white,fill=green!30,opacity=0.6] (O) -- (A) -- (C) -- (B) -- cycle;
\draw[green,fill=green!60] (O) -- (A) ;
\draw[gray,fill=green!60,->] (A) -- (Aa) ;
\draw[blue,fill=red!60,->] (O) -- (Ba) ;
\draw (Aa) node[below] {$p_1$};
\draw (Ba) node[left] {$p_2$};
\end{tikzpicture}
\caption{Parameter space of ODE example. Blue: $p_1 = 0$, and
Green: $\{(p_1,p_2)\in\RR^2_{\geq0}\mid p_1\neq 0\}$. }
\label{fig:ode}
\end{figure}
Example \ref{eg-ODESoln1} is an exception. In general one cannot so easily find an exact solution to an ODE system. Nevertheless, describing the equivalence classes can still be possible. Indeed, there are various approaches to building what is called in the literature an exhaustive summary (see for example \cite{otc-jrb-ebc:MethodsForIdentifiability}). An \emph{exhaustive summary} is simply a (not necessarily finite) collection $E$ of functions $P\to \RR$ that makes the model-data equivalence relation effective, that is, $p\sim\!\!_{M}p'$ if and only if $f(p)=f(p')$ for all $f\in E$. The differential algebra approach, introduced by Ljung and Glad \cite{Ljung} and Ollivier \cite{Ollivier}, relies on using exhaustive summaries. For an ODE system with time series data given by rational functions, one derives an input-output equation whose coefficients (once normalized so that the first term is one) provide an exhaustive summary for a dense open subset of parameter space.

Additional details on exhaustive summaries are given by Ollivier \cite{Ollivier} and Meshkat et al. \cite{Meshkat:2011is}, and software is available for computing input-output equations \cite{DAISY}. Exhaustive summaries are useful for determining identifiability (subsequently defined) and finding identifiable parameter combinations.


\subsection{Structural Identifiability}\label{Subsection-Identifiability}

We formulate a definition of structural identifiability in terms of the model-data equivalence relation defined at the beginning of this section. We base our rigorous understanding of the various flavors of identifiability in Sullivant's in-progress book on Algebraic Statistics \cite{ss:as} and Di Stefano III's book on Systems Biology \cite{DiStefanoIII-book}.

\begin{defn}[Structural Identifiability]
Let $(M,z)$ be a mathematical model with a choice of perfect data $z$ inducing an equivalence relation $\sim\!\!_{M,z}$ on the parameter space $P$.
\begin{itemize}
\item The pair $(M,z)$ is \emph{globally identifiable} if every equivalence classe consists of a single element.

\item The pair $(M,z)$ is \emph{generically identifiable} if for almost all $p\in P$, the equivalence class of $p$ consist of a single element.

\item The pair $(M,z)$ is \emph{locally identifiable} if for almost all $p\in P$, the equivalence class of $p$ has no accumulation points.

\item The pair $(M,z)$ is \emph{non-identifiable} if at least one equivalence class contains more than one element.

\item The pair $(M,z)$ is \emph{generically non-identifiable} if for almost all $p\in P$, the equivalence class of $p$ has accumulation points (or are positive dimensional). 
\end{itemize}
\end{defn}

\begin{rmk}
 In the definition above, ``almost all'' is used to mean that the property holds on a dense open subset of parameter space with respect to the usual Euclidean topology on $\RR^r\supset P$. Recall also that $q\in Q\subseteq P\subseteq \RR^r$ is an accumulation point of $Q$ if every open neighborhood of $p$ contains infinitely many elements of $Q$.
\end{rmk}

\begin{rmk}
In the ODE systems literature, where local identifiability is the main concern, ``non-identifiable'' is often used to mean what we have called ``generically non-identifiable''. \end{rmk}

\begin{eg}[The sum of exponentials]\label{eg-SumExponentials2}
We revisit Example \ref{eg-SumExponentials1} where we computed the equivalence class of any parameter $(a,b)\in\RR_{\geq 0}$. We found that $[(a,b)]=\{(a,b),(b,a)\}$, where $[(a,b)]$ denotes the set of parameters equivalent to $(a,b)$. It follows that this model is not globally identifiable, and so non-identifiable. This model is locally identifiable since every equivalence class has size at most 2.\done\end{eg}

\begin{eg}[Two biased coins]\label{eg-2BiasedCoins2}
We revisit the model considered in Example \ref{eg-2BiasedCoins1}. We showed that the equivalence class of a parameter $(p_1,p_2,p_3)\in[0,1]^3$ is
\begin{align*}
&\{(p_1,p_2,p_3),(1-p_1,p_3,p_2)\}     & \text{ if }p_1\neq 0,1,1/2~p_2\neq p_3\\
&\{(q,p_2,p_2) \mid q\in [0,1]\}              & \text{ if } p_2=p_3\\
&\{(0, q, p_3) \mid q\in [0,1] \}\cup \{(1,p_2,q) \mid q\in [0,1]\}  & \text{ if } p_1=0,1,1/2 ~ p_2\neq p_3,\\
&\{(1/2,p_2,p_3)\}                                    & \text{ if } p_1=1/2.
\end{align*}
This model is not globally identifiable (in fact no equivalence class is a singleton), but it is locally identifiable. Indeed, the equivalence classes have size two for almost all values of the parameter; only the parameters in the 2-dimensional subset $\{(p_1,p_2,p_3) \mid p_1(p_1-1)(p_2-p_3)=0\}$ have positive dimensional equivalence classes.\done\end{eg}

\begin{eg}[Fitting points to a line]\label{eg-Line2}
We revisit the model discussed in Example \ref{eg-Line1}. We saw that $[(a_0,a_1)]=\{(a_0,a_1)\}$ for all possible values of the parameter, therefore this model is globally identifiable.\done\end{eg}

\begin{eg}[An ODE system with an exact solution]\label{eg-ODESoln2}
For the model studied in Example \ref{eg-ODESoln1}, the equivalence class of a parameter $p=(p_1,p_2)\in P=\RR^2_{\geq0}$ is
\begin{align*}
&\{(p_1,p_2)\}                                      & \text{ if } p_1\neq 0,\\
&\{(0,q) \mid q\in \RR_{\geq 0}\} & \text{ if }p_1=0.
\end{align*}
As some equivalence classes are infinite, this model is not globally identifiable, but it is generically identifiable. Indeed, the equivalence classes of parameters belonging to the dense open subset $\{(p_1,p_2)\in P \mid p_1\neq 0\}$ have size 1.\done\end{eg}

\begin{eg}[{A nonlinear ODE model, see \cite[Example 6]{gm-er-mjc-hpw:DiffAlgIdent} and \cite[Example 5]{nm-me-jjd3:CombosAlgo}}]\label{eg-NonlinODE}
We now consider a model given by an ODE system with time series data and describes the behavior of a variable $(x_1,x_2)$ depending on a $5$-dimensional parameter $(p_1,p_2,p_3,p_4,p_5)$ with measurable output $y=x_1$. The ODE system is given by:
\begin{align*}
\dot{x}_1 =&p_1x_1-p_2x_1x_2\\
\dot{x}_2 =&p_3x_2(1-p_4x_2)+p_5x_1x_2
\end{align*}
The differential algebra method produces an exhaustive summary
\begin{align*}
\phi_1=\frac{p_3p_4}{p_2}-1,~ \phi_2=\frac{-2p_1p_3p_4}{p_2}  - p_3,~ \phi_3=-p_5,~\phi_4=\frac{p^2_1p_3p_4}{p_2} + p_1p_3,~ \phi_5=p_1p_5.
\end{align*}
That is, there is a dense open subset $U\subseteq P$ on which the model-data equivalence relation coincides with the equivalence relation given by the map
\begin{alignat*}{3}
\phi&\colon& U&\to &&\RR^4\\
               & &(p_1,p_2,p_3,p_4,p_5)&\mapsto && (\phi_1,\phi_2,\phi_3,\phi_4,\phi_5)
\end{alignat*}
We may take $U$ to be the set of parameters such that all $p_2$, $p_5$ and $2p_2+p_2p_3+p_1p_2-4p_1p_3p_4$ are nonzero. Then for $(p_1,p_2,p_3,p_4,p_5)\in U$, we have
\begin{align}\label{eqn-phi-inv1}
p_1&= -\frac{\phi_5}{2\phi_3}, \quad \quad \quad 
p_3=-2-\phi_2- \frac{2\phi_1\phi_5}{\phi_3},\\
\frac{p_4}{p_2}&=\frac{\phi_3(1+\phi_1)}{-2\phi_3-\phi_2\phi_3-2\phi_1\phi_5}, \quad \quad \quad 
p_5= \phi_5. \label{eqn-phi-inv2}
\end{align}
Let $\rho\colon U\to \RR^4$ be the map given by $(p_1,p_2,p_3,p_4,p_5)\mapsto (p_1,p_3,p_4/p_2,p_5)$. The map $\phi$ factors through $\rho$, and the formulas \eqref{eqn-phi-inv1},\eqref{eqn-phi-inv2} above provide an inverse for the induced function $\phi\colon \rho(U)\to \phi(U)$, and so in particular this function is bijective. It follows that for $(p_1,p_2,p_3,p_4,p_5)\in U$ the function  $\rho$ determines the model-data equivalence relation. Therefore, the equivalence class of $(p_1,p_2,p_3,p_4,p_5)\in U$ is
\[\left\{ \left(p_1,q_1,p_3,\frac{p_4}{p_2}\cdot q\right) \sth q\in \RR\right\}.\]
Hence, all parameters in $U$ have a 1-dimensional equivalence class and we conclude that the model is generically non-identifiable.\done\end{eg}

The main strategy we employed in the above examples was to construct a map $\phi\colon P\to\RR^N$ for some $N$, such that $p\sim\!\!_{M,z} p'$ if and only if $\phi(p)=\phi(p')$, that is, a map making the equivalence relation $p\sim\!\!_{M,z} p'$ effective. The model we considered was given in this way, or we evaluated an explicit time dependent model (or a solution to an ODE model) at finitely many timepoints, or else we used an alternative method to obtain an exhaustive summary and thus such a map. When the model-data equivalence relation can be made effective via a differentiable map $f\colon P\to \RR^N$, that is, when we can find $f$ such that $\sim\!\!_{M,z} =~\sim_f$, it is also possible to determine the local identifiability of the model by looking at the Jacobian of $f$. The model is locally identifiable if and only if the Jacobian of $f$  has full rank for generic values of $p$. Indeed, this is an immediate consequence of the Inverse Function Theorem. This method is regularly employed in algebraic statistics when considering specific models (see for example \cite[Proposition 15.1.7]{ss:as}).

In the case of ODE systems for which we do not have a solution and are unable to obtain an exhaustive summary, there are computational methods for establishing the (local) identifiability, see e.g. \cite{otc-jrb-ebc:MethodsForIdentifiability}, \cite{ar-jk-mps-mj-jt:comparisonIdentifiability} for a survey of the techniques available.


\section{Model Predictions}\label{Section-ModelPredictions}

In this section we provide a rigorous definition and a more mathematically correct name for ``model manifold'', a geometric object that takes center stage in the sloppiness literature \cite{Transtrum:2010ci,Transtrum:2011de,Transtrum:2014hr,Transtrum:2015hm}.

\begin{defn}
Let $M$ be a mathematical model with parameter space $P$ and a choice of perfect data. Suppose that the perfect data produced for each parameter value $p\in P$ is a point of $\RR^N\!$ for some $N$. A  \emph{model prediction map} is a map $\phi\colon P\to \RR^N$ that expresses the perfect data produced for the parameter value $p$ as a function $\phi(p)$.\end{defn}

A model prediction map is a geometric  realization of the quotient $P/\!\!\sim\!\!_{M,z}$ in the sense that it factors through the set-theoretic quotient $P\to P/\!\!\sim\!\!_{M,z}$ in such way that the induced map $\overline{\phi}\colon P/\!\!\sim\!\!_{M,z}\to \RR^N$ is injective.

A model prediction map is meant to be more than just a map making the model-data equivalence relation effective: we want to use this map to perform parameter estimation by finding the nearest model prediction (in the image of $\phi$) to a given noisy data point (in the data space, possibly off the image of $\phi$). 

\begin{rmk}
The sloppiness literature uses the term ``model manifold'' for the image of a model prediction map \cite{Transtrum:2010ci,Transtrum:2011de,Transtrum:2014hr,Transtrum:2015hm}. Although in general the image of $\phi$ is not a manifold as such, using the term manifold has the benefit of bringing into focus the geometric structure of mathematical models.
\end{rmk}

\begin{rmk}
Note that we do not require a model prediction map to satisfy the universal property of a categorical quotient, that is, we do not require that any map that is constant on the equivalence class factors through $\phi$. 
\end{rmk}

Each fiber of $\phi$ is a single equivalence class. As a consequence, when there is a model prediction map, then $\sim\!\!_{M,z} =~~\sim\!\!_\phi$, that is, the model-data equivalence relation coincides with the equivalence relation induced by $\phi$. Therefore,  identifiability can be characterized in terms of model prediction maps:


\begin{prop}
Let $M$ be a mathematical model and suppose there is a model prediction map $\phi\colon P\to \RR^N$ for some $N>0$. Then
\begin{itemize}
\item The pair $(M,\phi)$  is \emph{globally identifiable} if $\phi$ is injective.

\item The pair $(M,\phi)$  is \emph{generically identifiable} if $\phi$ is generically injective.

\item The pair $(M,\phi)$  is \emph{locally identifiable} if almost all non-empty fibers of $\phi$ have no accumulation points.

\item The pair $(M,\phi)$  is \emph{non-identifiable} if $\phi$ is not injective.

\item The pair $(M,\phi)$  is \emph{generically non-identifiable} if almost all non-empty fibers of $\phi$ have accumulation points.
\end{itemize}
\end{prop}

In some situations, it may be possible to construct a model prediction map only on a dense open subset of parameter space. A subset $E\subseteq P$ is \emph{$\sim\!\!_{M,z}$-stable} if $p\in E$ and $p'\sim\!\!_{M,z} p$ implies $p'\in E$, that is, $E$ is the union of equivalence classes.

\begin{defn}
 A \emph{generic model prediction map} is a model prediction map $\varphi\colon U\to \RR^N$ that is defined on a $\sim\!\!_{M,z}$-stable dense open subset $U\subseteq P$ of parameter space. 
\end{defn}

We will use the notation  $\varphi\colon P\dashrightarrow \RR^N$ borrowed from rational maps in the algebraic category to denote generic model prediction map when the exact domain of definition is unknown or not important. Three of the above notions of identifiability can be rephrased in terms of generic model prediction maps:

\begin{prop}
Let $M$ be a mathematical model and suppose there is a generic model prediction map $\varphi\colon P\dashrightarrow \RR^N$ for some $N>0$. Then
\begin{itemize}

\item The pair $(M,\varphi)$ is \emph{generically identifiable} if $\varphi$ is injective on its domain of definition.

\item The pair $(M,\varphi)$  is \emph{locally identifiable} if almost all non-empty fibers of $\varphi$ have no accumulation points.

\item The pair $(M,\varphi)$  is \emph{generically non-identifiable} if almost all non-empty fibers of $\varphi$ have accumulation points.
\end{itemize}
\end{prop}

In the algebraic category, we have an additional notion of identifiability:

\begin{defn}[Rational Identifiability] 
Let $(M,\phi)$ (resp. $(M,\varphi)$) be a mathematical model with and algebraic model prediction map defined over $\RR$ (resp. a generic model prediction map given by a rational map with real coefficients). We say that  $(M,\phi)$  (resp. $(M,\varphi)$) is \emph{rationally identifiable} if and only if each parameter $p_j$ can be written as a rational function of the $\phi_i$'s (resp. the $\varphi_i$'s), or equivalently if the fields of rational functions are equal: $\RR(p_1,\ldots,p_r)=\RR(\phi_1,\ldots,\phi_n)$ (resp. $\RR(p_1,\ldots,p_r)=\RR(\varphi_1,\ldots,\varphi_n)$).
\end{defn}

Note that rational identifiability implies generic identifiability. The implication is strict because we are working over a non-algebraically closed field (i.e. $\RR$).

\begin{eg}[An example of global identifiability, but not rational identifiability]
Consider the model $M$ with model prediction map $\phi:\RR\to \RR$ defined on the parameter space $\RR$ by $p\mapsto p^3+p$. First, we show that $M$ is globally  identifiable. Let $a$ and $b$ be two real numbers such that $a^3+a=b^3+b$. We can rewrite  $a^3+a=b^3+b$ as $(a-b)(a^2+ab+b^2+1)=0$. The polynomial function $a^2+ab+b^2+1$ has no real zeros, since for any given $b\in \RR$, it is a polynomial of degree 2 in $a$ with discriminant $-3b^2-4<0$. It follows that $a=b$, and so the model is globally identifiable. As $x$ is not a rational function of $x^3+x$, $(M,\phi)$ is not rationally identifiable.\done\end{eg}

The case of finite discrete parametric statistical models is again the simplest case, since the parameterization map is a model prediction map.  
For the two biased coin model studied in Examples \ref{eg-2BiasedCoins1} and \ref{eg-2BiasedCoins2}, the map $\phi$ is a model prediction map. It is possible to have non-isomorphic sets of model predictions, and also, as in the following example, we may have model prediction maps  belonging to different categories (real-analytic vs algebraic).

\begin{eg}[Gaussian Mixtures]\label{eg-GaussianMixtures1}
We consider the mixture of two 1-dimensional Gaussians, a model that can be used to describe the behavior of one measurement we make on individuals belonging to two populations. The model goes back to Pearson in 1894 who developed the methods of moments while studying crabs in the Bay of Naples. We follow the treatment by Am\'endola, Faug\`ere and Sturmfels \cite{Amendola:2015um}. The parameter is 5-dimensional: $(\lambda,\mu,\sigma,\nu,\tau)\in[0,1]\times \RR\times \RR_{\geq 0}\times \RR\times \RR_{\geq 0}=:P$. The mixing parameter $\lambda$ gives the proportion of the first population, the remaining four coordinate parameters are the means and variances of the two Gaussian distributions: $\mu,\sigma$ and $\nu,\tau$. Note that this model is at best locally identifiable. Indeed, since we cannot tell to which population an individual belongs, the parameters $(\lambda,\mu,\sigma,\nu,\tau)$ and $(1-\lambda,\nu,\tau,\mu,\sigma)$ will induce the same probability distribution (that is, the same perfect data) and so we will have $[(\lambda,\mu,\sigma,\nu,\tau)]\supseteq\{(\lambda,\mu,\sigma,\nu,\tau),(1-\lambda,\nu,\tau,\mu,\sigma)\}$, that is, the equivalence class of a parameter includes its orbit under an affine action of the symmetric group on two elements. It follows that generic equivalence classes will have size at least 2. Non-generic special cases will include the case where both populations have the same behavior, that is, $(\mu,\sigma)=(\nu,\tau)$, and the case where only one population is actually present, that is, $\lambda=0$ or $\lambda=1$. In these cases the equivalence class of a parameter contains certain subsets as follows:
\begin{align*}
[(\lambda,\mu,\sigma,\mu,\sigma)]&\supseteq\{(q,\mu,\sigma,\mu,\sigma) \mid q\in[0,1]\} &&& \text{if }(\mu,\sigma)=(\nu,\tau)\\
[(0,\mu,\sigma,\nu,\tau)]&\supseteq \{(0,q_1,q_2,\nu,\tau),(1,\nu,\tau,q_1,q_2)\mid q_1\in \RR,~q_2\in\RR_{\geq0}\} &&& \text{if } \lambda=0\\
[(1,\mu,\sigma,\nu,\tau)]&\supseteq \{(1,\mu,\sigma,q_1,q_2),(0,q_1,q_2,\mu,\sigma)\mid q_1\in \RR,~q_2\in\RR_{\geq0}\} &&& \text{if } \lambda=1
\end{align*}
In particular, some non-generic equivalence classes will be 1 and 2-dimensional.

As well as a cumulative distribution function $F(x)$, this model has both a probability density function $f(x)$ and a moment generating function $M(t)$; either characterizes the model. The 
probability density function is the map
\begin{alignat*}{3}
f&\colon & \RR\times P &\to&& \RR\\
               &&  x &\mapsto&& \lambda\left( \frac{1}{\sigma\sqrt{2\pi}}e^{-\frac{(x-\mu)^2}{2\sigma^2}}\right)+
                                              (1- \lambda)\left( \frac{1}{\tau\sqrt{2\pi}}e^{-\frac{(x-\nu)^2}{2\tau^2}}\right),
\end{alignat*}
the cumulative distribution function is the map
\begin{alignat*}{3}
F&\colon & \RR\times P &\to&& \RR\\
               &&  x &\mapsto&& \lambda\left( \frac{1}{\sigma\sqrt{2\pi}}\int_{\infty}^x e^{-\frac{(t-\mu)^2}{2\sigma^2}}dt \right)+
                                              (1- \lambda)\left( \frac{1}{\tau\sqrt{2\pi}}\int_{\infty}^xe^{-\frac{(t-\nu)^2}{2\tau^2}}dt\right),
\end{alignat*}
and the moment generating function is
\[M(t)=\sum_{i=0}^\infty \frac{m_i}{i!}t^i=\lambda e^{\mu t+\sigma^2t^2/2}+(1-\lambda) e^{\nu t+\tau^2t^2/2}.\]
Note that the $m_i$'s are polynomial maps in the five parameters and $M(t)$ is defined on some interval $(-a,a)$.  Thus $M$ can be seen as a function $M\colon (-a,a)\times P \to \RR$. The statement that these three functions characterize the distribution means that the equivalence relations they induce on $P=[0,1]\times \RR_{\geq 0}^4$ coincide with the model-data equivalence relation. By the $2r+1$ result \cite{Sontag:2rplus1}, it follows that for generic $x_1,\ldots,x_{11}$ and $t_1,\ldots,t_{11}$ each of the functions
\begin{alignat*}{3}
\phi_1 &\colon & P &\to&& \RR^{11}\\
                          && p &\mapsto&& (f(x_1,p),\ldots,f(x_{11},p)),
\end{alignat*}
\begin{alignat*}{3}
\phi_2&\colon & P& \to&& \RR^{11}\\
                         & & p &\mapsto&& (F(x_1,p),\ldots,F(x_{11},p)),
\end{alignat*}
and
\vspace{-0.2cm}
\begin{alignat*}{3}
\phi_3&\colon & P& \to&& \RR^{11}\\
                   &  & p& \mapsto && (M(t_1,p),\ldots,M(t_{11},p))
\end{alignat*}
also induce the model-data equivalence relation. Let $X_1,...,X_K$ denote a random sample from the distribution. As the moment generating function can be estimated from the sample via $\frac{1}{K}\sum_{i=1}^K e^{tX_i}$, the map $\phi_3$ is a model prediction map. As the cumulative distribution map can be estimated by the empirical distribution function, $\phi_2$ is also a model prediction map. The probability density function can also, in principle, be indirectly estimated from the sample by numerically deriving the empirical distribution function that estimates the cumulative distribution function. Thus $\phi_1$ can also be also be considered as a model prediction map.

This model also has algebraic model prediction maps. Indeed, the set of moments $\{m_i \mid i\geq 0\}$ determines $M$, which implies that this set of polynomial functions $P\to \RR$ will also induce the model-data equivalence relation. To obtain an algebraic model prediction map it will suffice to find a finite separating set $E\subset \RR[m_i \mid i\geq 0]\subseteq \RR[\lambda,\mu,\sigma,\nu,\tau]$, that is, a set $E$ such that whenever two points of $\RR^5$ are separated by some $m_i$, there is an element of $E$ that separates them (see \cite{gk:si} for a treatment of separating sets for rings of functions). As $\RR[\lambda,\mu,\sigma,\nu,\tau]$ is a finitely generated $\kk$-algebra, by \cite[Theorem 2.1]{gk:si} finite separating sets exist, and for $d$ large enough the first $d+1$ moments $m_0,m_1,\ldots,m_d$ will form a separating set. In fact, through careful algebraic manipulations it is possible to show that the first $7$ moments already form a separating set (see \cite[Section 3]{Amendola:2015um} or \cite{dl:InjectivityGaussianMixture}). As it is possible to estimate moments from data (via the sample moments $\frac{1}{K}\sum_{i=1}^K X_i^j$ for $j\geq 1$), we have a fourth model prediction map
\begin{alignat*}{3}
\phi_4&\colon & [0,1]\times \RR_{\geq 0}^4 &\to&& \RR^6\\
                     && p &\mapsto&& (m_1,m_2,m_3,m_4,m_5,m_6).
\end{alignat*}\done\end{eg}

Let $M$ be given by a real-analytic ODE system with time series data or an explicit time dependent model with time series data. Then by the $2r+1$ result \cite{Sontag:2rplus1}, we know that there exist model prediction maps that capture all the time series information. For the explicit models it is simply a matter of choosing timepoints. For ODE systems, we would in principle need an exact solution. First, some examples of explicit time dependent models:

\begin{eg}[Fitting points to a line]\label{eg-Line3}
By the discussion in Example \ref{eg-Line1}, the model-data equivalence relation coincides with the equivalence relation induced by evaluating the variable $x$ at the timepoints $t_1=0$ and $t_2=1$. As there is an invertible linear transformation taking any two distinct timepoints $(t_1,t_2)$ to $(0,1)$, any choice of two timepoints will give a model prediction map
\begin{alignat*}{3}
\phi_{t_1,t_2}&\colon &\RR^2&\to&&\RR^2\\
                     &  & (a_0,a_1)&\mapsto&& (a_0+t_1a_1,a_0+t_2a_1).
\end{alignat*}
Each corresponding set of model predictions, that is the image of $\phi_{t_1,t_2}$, actually fill up $\RR^2$. The set of model predictions we would obtain by taking more timepoints would still be isomorphic to $\RR^2$.\done\end{eg}

\begin{eg}[Sum of exponentials]\label{eg-SumExponentials3}
By the $2r+1$ result \cite{Sontag:2rplus1}, any generic choice of 5 timepoints will provide a model prediction map, but as we saw in Example~\ref{eg-SumExponentials2}, two timepoints suffice. As in the paper \cite{Transtrum:2011de}, we use the three timepoints $t_1=1/3,t_2=1,t_3=3$ to define a model prediction map
\begin{alignat*}{3}
\phi &\colon & \RR_{\geq 0}&\to&& \RR^3\\
                     && (a,b) &\mapsto&& (e^{-\nicefrac{a}{3}}+e^{-\nicefrac{b}{3}},e^{-a}+e^{-b},e^{-3a}+e^{-3b}).
\end{alignat*}
The image of $\phi $, the corresponding set of model predictions, is a surface with a boundary given by the image of the line $\{(a,b)\mid a=b\}$. A set of model predictions obtained by measuring at two timepoints  will consist of a closed subset of the positive quadrant of $\RR^2$.\done\end{eg}

For an ODE system with time series data, if we have an exact solution then we can easily construct a model prediction map as in the explicit time dependent case. In the absence of a solution, it may still be possible to construct a model prediction map, at least on a dense open subset of parameter space. For example, the coefficients of the input-output equations used in the differential algebra approach to obtain an exhaustive summary can be estimated from data (see for example \cite[p. 17]{fb:DiffElim}). Hence, in this case one can construct a rational model manifold. For example, in Example \ref{eg-NonlinODE} the map $\phi\colon U\to \RR^5$, when seen as a rational map on the whole parameter space, is a rational model prediction map. In general, however, the best one can do is solve the ODE system numerically and build a \emph{numerical model prediction map} as is done in the sloppiness literature \cite{Transtrum:2010ci,Transtrum:2011de,Transtrum:2014hr,Transtrum:2015hm}. A numerical model prediction map will provide some information on the model equivalence relation induced by an exact model prediction map; the quality of this information will depend on the quality of the numerics.


\section{Sloppiness and its relationship to  identifiability}\label{Section-Sloppiness}

We consider a model $M$ with a fixed choice of model prediction map $\phi$. A similar analysis can be made for a model with a generic model prediction map $\varphi$ by replacing $P$ with the domain of definition of $\varphi$ where needed. For the rest of this paper we focus on models with model prediction maps. 

We now consider the situation in which the data are model predictions corrupted by measurement noise with a known probability distribution. Hence, according to our assumption, the \emph{noisy data} is the result of a random process. We define the \emph{data space} $Z\subseteq \RR^N$ to be the set of points of $\RR^N$ that can be obtained as a corruption of the perfect data; how much it extends beyond the model predictions will depend on the support of the probability distribution of the measurement noise. The probability density function of the noisy data that can arise for the parameter value $p\in P$ is denoted by $\psi(p,\cdot)\colon Z  \to \RR$; it is the probability density of observing data $z \in Z$, which, for each $p\in P$, depends on the model prediction $\phi(p)$ rather than depending directly on the parameter $p$.

The \emph{Kullback-Leibler divergence}, used in probability and information theory, quantifies the difference between two probability distributions \cite{sk-ral:klDiv}. We define a premetric on parameter space via the Kullback-Leibler divergence:
\begin{align}\label{eq:tildeDformula}
d(p,p') &:= \int_{Z} \psi(p,z) \log\left(\frac{\psi(p,z)}{\psi(p',z)} \right) dz.
\end{align}
Gibb's Inequality \cite{tmc-jat:infTheory} proves that the Kullback-Leibler divergence is nonnegative, and zero only when the two probability distributions are equal on a set of probability one. It follows that $d$ is a \emph{premetric}, that is, $d(p,p')\geq 0$ and $d(p,p)=0$. Furthermore, $d(p,p')=0$ if and only if the probability distributions $\psi(p,\cdot)$ and $\psi(p',\cdot)$ are equal on a set of probability one, which is equivalent to $\phi(p)=\phi(p')$, since the dependance of $\psi$ on $p$ is only via the model prediction $\phi(p)$. Note that in general the Kullback-Leibler divergence and the premetric $d$ are not symmetric and do not satisfy the triangle inequality.

\begin{eg}[The case of additive Gaussian measurement noise]\label{eg-AdditiveGaussianNoise1}
Suppose the observations of a model prediction are distributed as follows:
\begin{align}
z \sim \mathcal{N}(\phi(p), \Sigma), \label{eq:gaussianNoise}
\end{align}
where  $\mathcal{N}(\phi(p), \Sigma)$ denotes a multivariate Gaussian distribution with mean $\phi(p) \in \RR^N$ and covariance matrix $\Sigma$, a $N\times N$ positive semi-definite matrix. This is equivalent to specifying that $z=\phi(p)+\epsilon$ where $\epsilon \sim \mathcal{N}(0, \Sigma)$, that is, the measurement noise is additive and Gaussian. We let $K$ be the number of experimental replicates, or the size of the sample. The density of a multivariate Gaussian then gives $\psi(p,\cdot)$ as
\begin{align*}
&\psi(p,z) = \left(2\pi\right)^{-\frac{NK}{2}} |\Sigma|^{-\frac{K}{2}} \exp \left( -\frac{K}{2} \big\langle \left(z - \phi(p) \right), \Sigma^{-1} \big(z - \phi(p) \big) \big\rangle \right),
\end{align*}
where $\langle \cdot,\cdot\rangle$ denotes the inner product.
The computation of \eqref{eq:tildeDformula}  then yields :
\begin{align}
&d(p,p') =   \frac{K}{2} \Big\langle \phi(p) - \phi(p'), \Sigma^{-1} \big(\phi(p) - \phi(p')  \big) \Big\rangle, \label{eq:gaussianTildeD}
\end{align}
 (details provided in \cite{Duchi:LinAlgDerivations}). Thus $d(p,p')$ is a weighted sum of squares, and so it is symmetric and satisfies the triangle inequality, and hence $d$ is a pseudometric. In particular, if $\Sigma$ is the identity matrix, then $d$ is induced by half of the square of the Euclidean distance in data space. The pseudometric $d$ is a metric exactly when the model is globally identifiable, since then $d(p,p')=0\Leftrightarrow\phi(p)=\phi(p')\Leftrightarrow p=p'$.\done
 \end{eg} 

It is often possible to equip parameter space with a metric, a natural choice being the Euclidean metric inherited from the ambient $\RR^r$. For instance our model might be of a chemical reaction network, where the coordinates of the parameter correspond to the positive, real-valued rate constants associated with particular chemical reactions.  In this case, a reasonable choice of reference metric is the Euclidean distance between different points in the positive real quadrant. The reference metric on parameter space may not be Euclidean. For example, the natural metric on tree space that arises in Phylogenetics, the BHV metric, is  non-Euclidean \cite{BHV:Treespace2001}.

We can now offer a new precise, but qualitative, definition of sloppiness. We discuss two different quantifications in the following two sections:

\begin{defn}\label{defn-Sloppiness}
Let $(M,\phi,\psi,d_P)$ be a mathematical model with a choice of model prediction map, a specific assumption on the probability distribution of the noisy data, and a choice of reference metric on $P$. We say that $(M,\phi,\psi,d_P)$ is \emph{sloppy} at $p_0$ if in a neighborhood of $p_0$ the premetric $d$ diverges significantly from the reference metric on parameter space.
\end{defn}

\subsection{Infinitesimal Sloppiness}\label{subsection:inf-slop}

We first provide the generally accepted and original quantification of sloppiness found in the literature, which we explain in terms of our new qualitative definition of sloppiness (see Definition \ref{defn-Sloppiness}). The  sloppiness literature makes the implicit assumption that the reference metric on parameter space is the standard Euclidean metric, and we make the same assumption in this section.

Fix $p_0\in P$ and consider the map $d(\cdot,p_0)\colon P\to \RR_{\geq 0}$ mapping $p$ to $d(p,p_0)$. Suppose that $d(\cdot,p_0)$ is twice continuously differentiable in a neighborhood of $p_0$. By definition, $d(p_0,p_0) = 0$, and furthermore $p_0$ is a local minimum of $d(\cdot,p_0)$, implying a null Jacobian. Therefore an approximation of $d(p,p_0)$ for $p$ in a neighborhood of $p_0$ is given by the Taylor expansion
\begin{align}\label{eq:infTildeD}
d(p,p_0) = \frac{1}{2}\Big\langle (p-p_0), \left(\nabla_p^2d(p,p_0)\right)|_{p=p_0} (p-p_0) \Big\rangle + \mathcal{O}(\|(p-p_0)\|_2),
\end{align}
where $\|\cdot\|_2$ is the Euclidean norm and  $ \nabla_p^2d(p,p_0)$ is the Hessian of the function $d(p,p_0)$, that is, the matrix $\left(\frac{\partial^2}{\partial p_i\partial p_j}d(p,p_0)\right)_{i,j}$ of second derivatives with respect to the coordinate parameters. This Hessian evaluated at $p=p_0$ is known as the \emph{Fisher Information Matrix (FIM) at $p_0$}.

 Local minimality of $d(\cdot,p_0)$ at $p_0$ ensures that the matrix $\left(\nabla_p^2d(p,p_0)\right)|_{p=p_0}$ is positive semidefinite, and so the FIM at $p_0$ induces a pseudometric on parameter space
 \vspace{-0.2cm}
 \begin{alignat*}{3}
 d_{\rm{FIM},p_0}&\colon & P\times P &\to&& \RR_{\geq 0}\\
                               && (p,p')& \mapsto && \frac{1}{2}\Big\langle (p-p'), \left(\nabla_p^2d(p,p_0)\right)|_{p=p_0} (p-p')\Big\rangle.
 \end{alignat*}
Note that the pseudo-metric $d(\cdot,p_0)$ is not the Fisher Information metric. When the FIM is positive definite, the Fisher Information metric is the Riemannian metric induced by the FIM by computing the line integral of the geodesic linking two parameters $p,p'\in P$ \cite{Amari2007}.

\begin{eg}[The case of additive Gaussian measurement noise]\label{eg-AdditiveGaussianNoise2}
In the sloppiness literature, measurement noise is assumed Gaussian, as in Example \ref{eg-AdditiveGaussianNoise1}, and  for $K=1$ the FIM $\left(\nabla_p^2d(p,p_0)\right)|_{p=p_0}$ is known as the \emph{sloppiness matrix at $p_0$}. 
Explicitly, the sloppiness matrix is
\begin{align}
\left(\nabla_p^2d(p,p_0)\right)|_{p=p_0}=   \frac{1}{2} \left((\nabla_p \phi(p))|_{p=p_0}\right)^T \Sigma^{-1} \left((\nabla_p \phi(p))|_{p=p_0}\right), \label{eq:sloppinessFormula}
\end{align}
where $(\nabla_p \phi(p))|_{p=p_0}$ denotes the Jacobian of $\phi$ with respect to the coordinate parameters evaluated at $p=p_0$.\done
\end{eg}

\begin{rmk}[Structural identifiability and the FIM]
The FIM is intimately linked to structural identifiability. Indeed, a result of Rothenberg \cite[Theorem 1]{Rothenberg} shows that
$M$ is locally identifiable if and only if the FIM is full rank at some $p_0$. If we assume additive Gaussian noise, then Equation \ref{eq:sloppinessFormula} implies that the rank $r_0$ of the FIM at $p_0$ is equal to the rank of the Jacobian of $\phi$ at $p_0$, and so for generic $p_0$, the dimension of the connected component of $p_0$ in its equivalence class is $r-r_0$ (cf discussion near  \cite[Equation 85]{cc-jjdIII:StructuralIdentifiabilityReview}). As one can compute the rank of the  FIM by computing the singular value decomposition and employing a sound threshold \cite{Donoho:2014}, the FIM can then be used to numerically determine the dimension of generic equivalence classes. Further approaches for giving probabilistic, and sometimes guaranteed bounds on identifiability using symbolic computation at specific parameters have been developed and applied in \cite{Anguelova2007,Sedoglavic2008,Anguelova2012}.
\end{rmk}
 
   The Taylor expansion \eqref{eq:infTildeD} shows that, for parameters very near $p_0$, the premetric $d$ is approximately given by the pseudometric $d_{\rm{FIM},p_0}$. Therefore, in a neighborhood of $p_0$, the map $\phi$ giving the model predictions is maximally sensitive to infinitesimal perturbations in the direction of the eigenvector of the maximal eigenvalue of the FIM at $p_0$, referred to as the \emph{stiffest} direction at $p_0$. The direction of the eigenvector of the minimal eigenvalue of the FIM at $p_0$, which gives the perturbation direction to which $\phi$ is minimally sensitive, is known as the \emph{sloppiest} direction at $p_0$.

\begin{defn}
Let $(M, \phi,\psi,d_2)$ be a mathematical model with a choice of model prediction map, a specific assumption on measurement noise, and the Euclidean metric as a reference metric on $P$. We say that $(M, \phi,\psi,d_2)$ is \emph{infinitesimally sloppy} at a parameter $p_0$ if there are several orders of magnitude between the largest and smallest eigenvalues of the FIM at $p_0$. We define the \emph{infinitesimal sloppiness at $p_0$} to be the condition number of the FIM at $p_0$, that is, the ratio between its largest and smallest eigenvalues.
\end{defn}

\begin{rmk}
First note that this definition is only meaningful when the FIM at $p_0$ is full rank. In this case, the condition number of the FIM at $p_0$ corresponds to the aspect ratio of the level curves of $d_{\rm{FIM},p_0}$, which is one way to quantify how far these level curves are from Euclidean spheres. Thus, using the condition number of the FIM as a quantification of sloppiness implies that the reference metric on $P$ is the Euclidean metric.
\end{rmk}

The FIM possesses attractive statistical properties. Suppose $(M, \phi,\psi,d_2)$ is locally identifiable and that maximum likelihood estimates exist generically, that is, for almost all $z$, there are parameters minimizing the negative log-likelihood:
$\hat{p}(z) = \min_{p \in P} (-\log \psi(p,z))$. Let $z\in Z$ be a generic data point and let $\hat{p}(z)$ be the unique maximum likelihood estimate. Suppose that the ``true'' parameter is $p_0$, that is, $z$ is a corruption of the model prediction $\phi(p_0)$. When the FIM at $p_0$ is invertible, the Cramer-Rao inequality \cite[Section 7.3]{gc-rlb:statsBook} implies that
\begin{align}
[\left(\nabla_p^2d(p,p_0)\right)|_{p=p_0} ]^{-1} \preceq \operatorname{Cov}_{p_0} \hat{p}(z), \label{eq:crLB}
\end{align}
where
\vspace{-0.2cm}
\begin{align*}
 \operatorname{Cov}_{p_0} \hat{p}(z) :=& \int_{Z} \hat{p}(z)\hat{p}^T(z) \psi(p_0,z) \ dz\\
                                  &\hspace{1.5cm}- \left(\int_{Z} \hat{p}(z) \psi(p_0,z) \ dz\right) \left(\int_{Z}  \hat{p}(z) \psi(p_0,z) \ dz \right)^T
\end{align*}
is the covariance of the maximum likelihood estimate with respect to measurement noise, and $A\preceq B$ if and only if $B-A$ is positive semi-definite. This inequality provides an explicit link between the uncertainty associated with parameter estimation and the geometry of the negative log likelihood.  Meanwhile, the sensitivity of $\phi$ is related to the uncertainty associated with parameter estimation via  \eqref{eq:crLB}. The asymptotic normality of the maximum likelihood estimates implies that the Cramer-Rao inequality \eqref{eq:crLB} tends to \emph{equality} as $K$ tends to infinity \cite[Section 10.7]{gc-rlb:statsBook}. Formally,
\begin{align}
\lim_{K \to \infty} [\left(\nabla_p^2d(p,p_0)\right)|_{p=p_0}]^{-1} =   \operatorname{Cov}_{p_0} \hat{p}(z). \label{eq:sloppinessValid}
\end{align}
The list of regularity conditions required for \eqref{eq:crLB} and \eqref{eq:sloppinessValid} to hold are provided in  \cite[Section 7.3]{gc-rlb:statsBook}, and are easily satisfied in practice.

\begin{rmk}
A sufficient condition for $d_{\rm{FIM},p_0}$ to be a good approximation for the premetric $d$ on a neighborhood of $p_0$ is to have a very large number of replicates. In practice, however, questions of cost and time mean that the number of replicates is often very small. Accordingly, the sloppiness literature generally assumes the number of experiments is one ($K=1$), though the effect of increasing experimental replicates in mitigating sloppiness has been explored in \cite{Apgar:2010di}.
\end{rmk}

\begin{eg}[Fitting points to a line]\label{eg-Line4} 
We revisit once more the model first considered in Example \ref{eg-Line1}. We consider the model prediction map obtained by evaluating at timepoints $t_1=0$ and $t_2=1$ as in Example \ref{eg-Line3}. We assume that we are in the presence of additive Gaussian error with covariance matrix $\Sigma=I_2$ equal to the identity matrix as discussed in Examples \ref{eg-AdditiveGaussianNoise1} and \ref{eg-AdditiveGaussianNoise2}. As in Example \ref{eg-AdditiveGaussianNoise1} the premetric $d$ is induced by half the square of the Euclidean distance on the data space $\RR^2$. We can explicitly determine $d$:
\begin{align*}
d((a_0,a_1),(a'_0,a'_1))&=\frac{1}{2}\left( (a_0-a'_0)^2+(a_0+a_1-a'_0-a'_1)^2\right)\\
                                                &=\frac{1}{2}\left(2(a_0-a'_0)^2+2(a_0-a'_0)(a_1-a'_1)+(a_1-a'_1)^2\right)\\
                                                &=\frac{1}{2} \left\langle \begin{pmatrix} a_0 -a'_0 \\ a_1-a'_1\end{pmatrix}, \begin{pmatrix} 2 & 1\\1 &1\end{pmatrix}\begin{pmatrix}a_0-a'_0\\a_1-a'_1\end{pmatrix}\right\rangle.
\end{align*}
We see that $d$ itself is a weighted sum of squares given by a positive definite matrix, and so $d$ is a metric. As the positive definite matrix giving this sum of squares is constant throughout parameter space, it follows that the sloppiness of the model is also constant throughout parameter space. Note that the same phenomenon would happen for any model such that the model manifold is given by an injective linear map (see Proposition \ref{prop:LinearPredictions} below). In particular, the same situation would arise when considering the problem of fitting points to any polynomial curve, as the corresponding model prediction map will be linear.

We next compute the FIM. The map $d(\cdot,(b_0,b_1))$ is given by 
\begin{align*}
d((a_0,a_1),(b_0,b_1))=\frac{1}{2}\left( (b_0-a_0)^2+(b_0+b_1-a_0-a_1)^2\right)
\end{align*}
and so its Hessian, that is, the FIM is

\begin{align*}
&\frac{1}{2}\begin{pmatrix}
\frac{\partial^2}{\partial a_0^2} d((a_0,a_1),(b_0,b_1))             & \frac{\partial^2}{\partial a_0\partial a_1}d((a_0,a_1),(b_0,b_1))\\
 \frac{\partial^2}{\partial a_1\partial a_0} d((a_0,a_1),(b_0,b_1)) & \frac{\partial^2}{\partial a_0^2} d((a_0,a_1),(b_0,b_1)) 
\end{pmatrix} =\begin{pmatrix}
2 & 1\\
1 & 1\end{pmatrix}
\end{align*}
We conclude that in this case, the pseudometric $d_{\text{FIM},(a_0,a_1)}$ coincides with $d$ on the entire parameter space, which we will see in Proposition \ref{prop:LinearPredictions} is a consequence of the linearity of the model prediction map.
\done\end{eg}

\begin{prop}\label{prop:LinearPredictions}
Let $(M,\phi, \mathcal{N}(\phi(p),\Sigma),d_2)$ be a mathematical model with parameter space $P\subseteq\RR^r$, a choice of model prediction map, additive Gaussian noise with covariance matrix $\Sigma$, and the Euclidean metric as a reference metric. If the model prediction map $\phi\colon P\to\RR^N$ is linear, then $d_{\text{FIM},p_0}=d$ for all $p_0\in P$.
\end{prop}
\begin{proof}
Our assumption that $\phi$ is linear implies that there is a $N\times r$ matrix $A$ with real entries such that $\phi(p)=Ap$. By the discussion in Example \ref{eg-AdditiveGaussianNoise1}, we have
\begin{align*}
d(p',p)&=\frac{K}{2}\Big\langle(Ap'-Ap),\Sigma^{-1}(Ap'-Ap)\Big\rangle\\
          &=\frac{K}{2}\Big\langle(A(p'-p)),\Sigma^{-1}(A(p'-p))\Big\rangle\\
          &=\frac{K}{2}\Big\langle(p'-p),(A^T\Sigma^{-1} A)(p'-p)\Big\rangle.\\
\end{align*}
On the other hand the FIM is given by
\begin{align*}
\nabla_p^2d(p,p_0) &= \nabla_p^2 \left( -\frac{NK}{2}\log(2\pi)+\frac{K}{2}\log(|\Sigma|)+\frac{K}{2}\Big\langle(Ap_0-Ap), \Sigma(Ap_0-Ap)\Big\rangle\right)\\
                                                               &=\frac{K}{2} \nabla_p^2 \left(\Big\langle(Ap_0-Ap),\Sigma^{-1}(Ap_0-Ap)\Big\rangle\right)\\
                                                               &= \frac{K}{2} \nabla_p^2 \left(\Big\langle(Ap),\Sigma^{-1}(Ap)\Big\rangle\right)\\
                                                               &=\frac{K}{2}A^T\Sigma^{-1} A,
\end{align*}
completing the proof.
\end{proof}

\begin{eg}[Linear parameter-varying model]\label{eg-linparvar}
We consider a standard model arising in control theory, which falls under the case of real analytic time 
dependent models. Specifically, we consider models of the form
\begin{alignat*}{3}
&\dot{x} &=& A(p)x,\\ 
&y &=& Cx, \\
&x(0) &= &x_0,
\end{alignat*}
where $A(p)$ is a $m\times m$ matrix with polynomial dependence on the parameter $p\in \RR^{r-1}$, $C$ is a known fixed $n\times m$ matrix with real coefficients, and $y$ is the measurable output. Note that $y$ depends on the initial condition $x_0$, which we will consider as an extension of parameter space. We assume further that $A(p)$ is Hurwitz for all $p$ considered. We denote by $y(t,(p,x_0))$ the output of the system at time $t$, given the parameter $(p,x_0)$.   If we measured the system at a finite number of time-points, assuming Gaussian noise-corruption, then the distance function $d((p,x_0),(p',x'_0))$ would be the Euclidean distance between the model predictions at the chosen set of timepoints. 

For any pair of parameters $(p,x_0)$ and $(p',x'_0)$, the following integral can be explicitly computed and is a rational function of $(p,x_0)$ and $(p',x'_0)$ (cf \cite[Theorem 1]{Raman2016}, which assumes that $A(p)$ is linear in the parameters, but whose proof holds more generally):
\begin{align*}
d_{\infty}((p,x_0),(p',x'_0)) := \int_{0}^{\infty} \|y(t,(p,x_0)) - y(t,(p',x'_0))\|_2^2   \ dt.
\end{align*}
Note that $d_{\infty}((p,x_0),(p',x'_0))$ is equal to the $L^2$ norm of the function $y(t,(p,x_0)) - y(t,(p',x'_0))$, and so $d_{\infty}((p,x_0),(p',x'_0))=0$ if and only if $y(t,(p,x_0))=y(t,(p',x'_0)$ for almost all $t$. As $y$ is real-analytic, it then follows that $y(t,(p,x_0))=y(t,(p',x'_0)$ for all $t$. Therefore $d_{\infty}((p,x_0),(p',x'_0))=0$ if and only if $(p',x'_0)\sim\!\!_{M,z}(p,x_0)$, and so the equivalence class of $(p,x_0)$ is given by the zeros $(p',x'_0)$ of the rational function $d_{\infty}((p,x_0),(p',x'_0))$. \done
\end{eg}


\subsection{Multiscale sloppiness}

We now present a quantification of sloppiness that holds for non-Euclidean reference metric and is better suited to the presence of noninfinitesimal noise. In this section, we sometimes make the assumption that for generic $p_0\in P$, there is a neighborhood of $p_0$ where the reference metric $d_P$ is  strongly equivalent to the Euclidean metric inherited by $P$ as a subset of $\RR^r$. The BHV metric \cite{BHV:Treespace2001} mentioned at the beginning of the section satisfies this property.

In Section \ref{subsection:inf-slop} we saw how the FIM approximates the premetric $d$ in the limit of decreasing magnitude of parameter perturbation, which is realizable in the limit of increasing experimental replicates or sample size. In a practical context, however the limit of increasing replicates may not be valid. Indeed examples are provided in \cite{keh-trm-rwa:nonlinIdent} and \cite{mf-asm-ak:bootstrap} of models for which the uncertainty of parameter estimation is poorly approximated by the FIM. Even when the approximation is valid, numerical errors in sloppiness quantification are often significant, due to the ill-conditioning of the FIM \cite{Vallisneri2008}. We describe a second approach called  \emph{multiscale sloppiness} introduced in \cite{dvr-ja-ap:SloppinessDhruva} for models given by ODE systems with time series data under the assumption of additive Gaussian noise and with the standard Euclidean metric as a reference metric. We extend this quantification of sloppiness to a more general setting.

\begin{defn}[Multiscale Sloppiness]
Consider a model $(M,\phi,\psi,d_P)$ with a choice of model prediction map $\phi$, a specific assumption on measurement noise, and a choice of reference metric on $P$. We define the \emph{$\delta$-sloppiness at $p_0$} to be
\begin{align*}
& \mathcal{S}_{p_0}(\delta)  := \frac{\sup_{p \in P} \{ d(p,p_0) \mid d_P(p,p_0) = \delta\}}{\inf_{p \in P} \{ d(p,p_0) \mid d_P(p,p_0) = \delta\}}
\end{align*}

If $d_P$ is strongly equivalent to the Euclidean metric on a neighborhood of $p_0$, then for $\delta$ sufficiently small, the (non-unique) \emph{maximally} and \emph{minimally} disruptive parameters at length scale $\delta$ at the point $p_0\in P$ are the elements of the sets
\begin{align}
&D^{max}_{p_0}(\delta) = \arg\max_{p \in P} d(p,p_0) : d_P(p,p_0) = \delta \label{eq:maxDisrupt} \\
&D^{min}_{p_0}(\delta) = \arg\min_{p \in P} d(p,p_0) : d_P(p,p_0) = \delta, \label{eq:minDisrupt}
\end{align}
respectively. 
In this case, the $\delta$-sloppiness at $p_0$ is
\begin{align}
& \mathcal{S}_{p_0}(\delta)  = \frac{d(p_{p_0}^{max}(\delta),p_0)}{d(p_{p_0}^{min}(\delta),p_0)},
\end{align}
where $p_{p_0}^{max}(\delta)\in D^{max}_{p_0}(\delta)$ and $d(p_{p_0}^{min}(\delta)\in D^{min}_{p_0}(\delta)$.
\end{defn}

Note that since the set $\{ d(p,p_0) \mid p\in P \text{ and } d_P(p,p_0) = \delta\}$ is a closed set of real numbers with a lower bound (zero), the infimum is actually a minimum, hence  $D^{min}_{p_0}(\delta)$ is always well-defined.

\begin{rmk}
Computation of $\delta$-sloppiness would seem to require the solution of a (possibly nonlinear, nonconvex) optimization program for each $\delta >0$. However, assuming the reference metric on parameter space is the Euclidean distance and that we are in the presence of additive Gaussian noise, finding $p^{min}_{p_0}(\delta)\in D^{min}_{p_0}(\delta)$ for continuous ranges of $\delta$ can be formulated as the solution of an optimal control problem relying on solving a Hamiltonian $dH/dp = 0$ as described in \cite[Section 5]{dvr-ja-ap:SloppinessDhruva}. With this method, computation of $\delta$-sloppiness is possible for large, nonlinear systems of ODE. Note that this formulation as an optimal control problem does not fundamentally rely on the assumption of a Euclidean metric on parameter space, and so the principle likely applies to more general classes of metric.
\end{rmk}

If we choose the usual Euclidean distance as the reference metric on parameter space, then as the length-scale $\delta$ goes to zero, infinitesimal and multiscale sloppiness coincide:
\begin{align*}
 \lim_{\delta \to 0} \mathcal{S}_{p_0}(\delta) = \frac{\lambda^{max}\left(\nabla_p^2d(p,p_0)\right)|_{p=p_0}}{\lambda^{min}\left(\nabla_p^2d(p,p_0)\right)|_{p=p_0}},
\end{align*}
where $\lambda^{max}$ and $\lambda^{min}$ denote the maximal and minimal eigenvalues of their argument. Indeed, the Taylor expansion (\ref{eq:infTildeD}) implies that as $p$ approaches $p_0$, $d(p,p_0)$ approaches $d_{\rm{FIM},p_0}(p,p_0)$, and so the level sets $\{p\in P \mid d(p,p_0)=\delta \} $ tend to the level sets $\{p \in P \mid d_{\rm{FIM},p_0}(p,p_0)=\delta\}$ as $\delta$ goes to zero.

Multiscale sloppiness, or more precisely the denominator of $\mathcal{S}_{p_0}(\delta)$, is closely related to structural identifiability:
\begin{thm}
Let $(M,\phi,\psi,d_P)$ be a mathematical model with a choice of model prediction map, a specific assumption on measurement noise, and a choice of reference metric $d_P$, which we assume is strongly equivalent to the Euclidean metric. The equivalence class $[p_0]_{\sim\!\!_{M,\phi}}$ of the parameter $p_0$ has size one if and only if $\inf_{p \in P} \{ d(p,p_0) \mid d_P(p,p_0) = \delta\} > 0$ for all $\delta>0$.
\end{thm}
\begin{proof}
Suppose that the equivalence class $[p_0]$ of $p_0$ has size one, then for any other parameter $p$, we will have $d(p,p_0)>0$. In particular, this will hold for $p^{min}_{p_0}(\delta)\in D^{min}_{p_0}(\delta)$, for any $\delta>0$. Hence $\inf_{p \in P} \{ d(p,p_0) \mid d_P(p,p_0) = \delta\}=d(p^{min}_{p_0}(\delta),p_0)>0.$

Suppose on the other hand that $\inf_{p \in P} \{ d(p,p_0) \mid d_P(p,p_0) = \delta\}  > 0$ for all $\delta>0$ and suppose, for a contradiction that $p\in[p_0]$ is distinct from $p_0$. Set $\delta':=d_P(p,p_0)$. As $p\neq p_0$, we have  $\delta'>0$ and 
\begin{align*}
0=d(p,p_0)\geq \inf_{p \in P} \{ d(p,p_0) \mid d_P(p,p_0) = \delta'\},
\end{align*}
which is a contradiction, since $\inf_{p \in P} \{ d(p,p_0) \mid d_P(p,p_0) = \delta'\}  > 0$.
\end{proof}

\begin{eg}[Sum of exponentials]\label{eg:SumExpInfDelta}

We highlight that sloppiness is a local property: it depends on the point in parameter space and the precise choice of timepoints. In this spirit, let us revisit Example \ref{eg-SumExponentials1}, again adding Gaussian measurement noise with identity covariance and taking the model prediction map to be evaluating at timepoints $\{1/3,1,3\}$. We are in the situation considered in Example \ref{eg-AdditiveGaussianNoise1} and so $d$ is again half the squared Euclidean distance between model predictions:
\vspace{-0.2cm}
\begin{align*}
d\left( \ (a,b), (a',b') \ \right) = \frac{1}{2} \|\phi(a,b) - \phi(a',b')\|_2^2,
\end{align*}
where
\vspace{-0.2cm}
\begin{align*}
\phi(a,b)  = (e^{-\nicefrac{a}{3}}+e^{-\nicefrac{b}{3}} ,e^{-a}+e^{-b},e^{-3a}+e^{-3b}).
\end{align*}
The Jacobian of the model prediction map at $(a_0,b_0)$ is therefore given as
\begin{align*}
(\nabla_{a,b}\phi(a,b))|_{(a,b)=(a_0,b_0)} = \begin{pmatrix}
-\frac{1}{3} e^{-\nicefrac{a_0}{3}}    & - e^{-a_0} & -3 e^{-3a_0} \\  -\frac{1}{3} e^{-\nicefrac{b_0}{3}}    & - e^{-b_0} & -3 e^{-3b_0}
\end{pmatrix}
\end{align*}
The FIM at $(a_0,b_0)$, in this case, will be given as 
\begin{align*}
\begin{pmatrix}
\frac{1}{9} e^{-\nicefrac{2a_0}{3}} + e^{-2a_0} + 9 e^{-6a_0} & \frac{1}{9} e^{-\nicefrac{(a_0+b_0)}{3}} + e^{-(a_0+b_0)} + 9 e^{-3(a_0+b_0)} \\
\frac{1}{9} e^{-\nicefrac{(a_0+b_0)}{3}} + e^{-(a_0+b_0)} + 9 e^{-3(a_0+b_0)} & \frac{1}{9} e^{-\nicefrac{2b_0}{3}} + e^{-2b_0} + 9 e^{-6b_0} 
\end{pmatrix}
\end{align*}
We compute infinitesimal sloppiness and $\delta$-sloppiness of $(M,\phi)$ at $p_0=(a,b)=(4, 1/8)$  in Figure~\ref{fig:multiscale}: the difference between these notions of sloppiness becomes clear.

 \begin{figure}[h!]
\includegraphics[width=5.75cm]{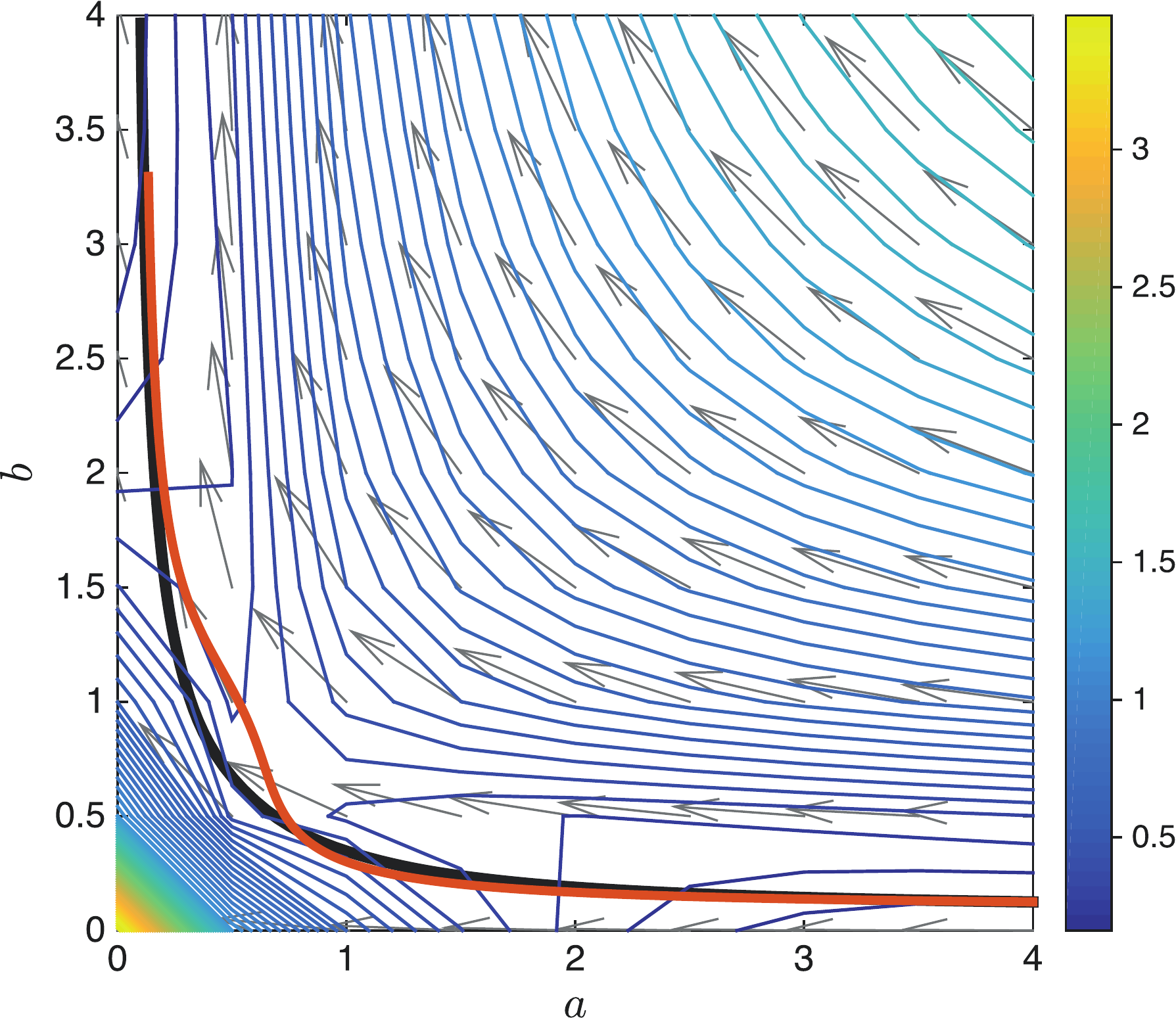}
\caption{Infinitesimal sloppiness vs $\delta$-sloppiness of sum of exponential (Example \ref{eg:SumExpInfDelta}). For a given parameter $p_0 = (a = 4, b = 1/8)$ and time points $\{1/3, 1, 3\}$,  level sets of $d$ are drawn (colors). The vector field consisting of the eigenvector corresponding to the largest eigenvalue of the FIM is plotted across the grid. We compare the flow of this vector field initialised at $p_0$ (gray curve), with the most delta-sloppy parameters with respect to $p_0$ over a range of $\delta$ (orange curve). }
\label{fig:multiscale}
\end{figure}

Figure~\ref{fig:VaryingTimepoints} illustrates how the change of model prediction map, in this case different choices of timepoints, changes the premetric $d$. This suggests that sloppiness should be taken into consideration when designing an experiment: some choices of timepoints will allow for better quality parameter estimation. Figure~\ref{fig:VaryingParameters}, on the other hand, illustrates how the premetric $d$ changes in parameter space. In particular, these two figures illustrate that unlike identifiability, sloppiness is not a global property of a model.\done

 \begin{figure}[h!]

\begin{tikzpicture}
\node    at  (-4.1,0) {\includegraphics[trim = 20mm 5mm 19mm 5mm, clip, width=4.5cm]{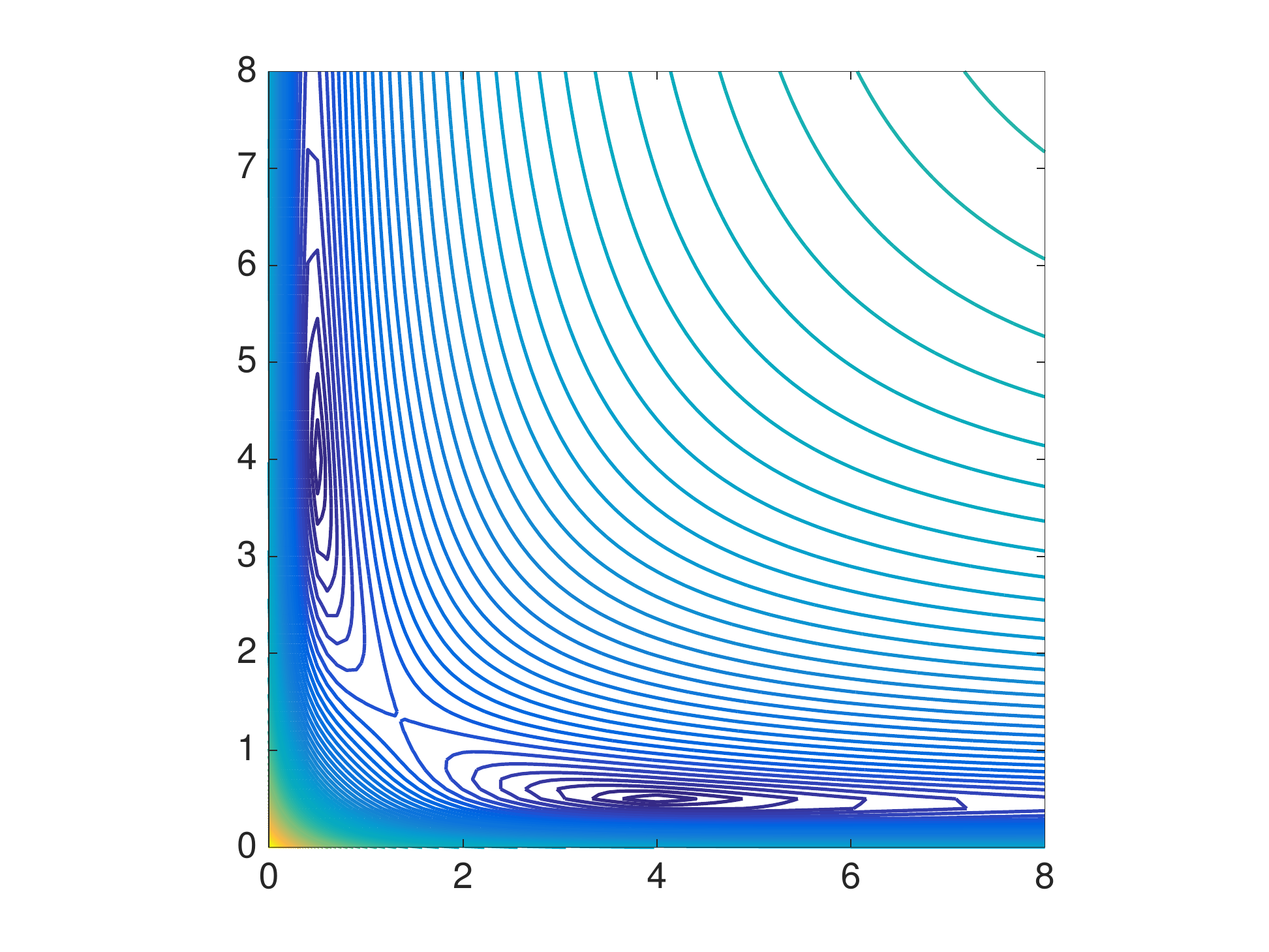}};

\node at (0,0) {\includegraphics[trim = 20mm 5mm 19mm 5mm, clip, width=4.5cm]{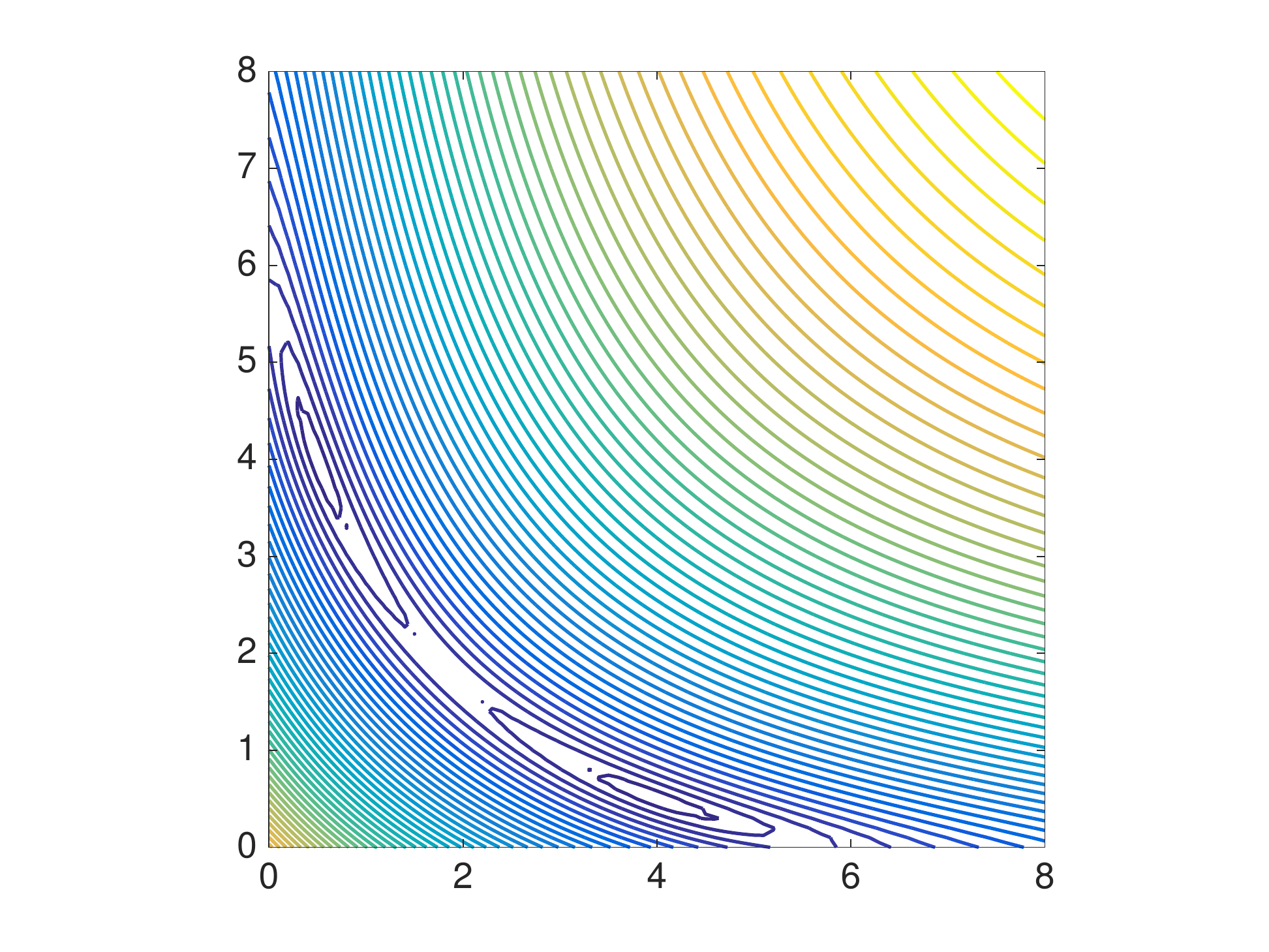}};

\node at (4.1,0) {\includegraphics[trim = 20mm 5mm 19mm 5mm, clip,  width=4.5cm]{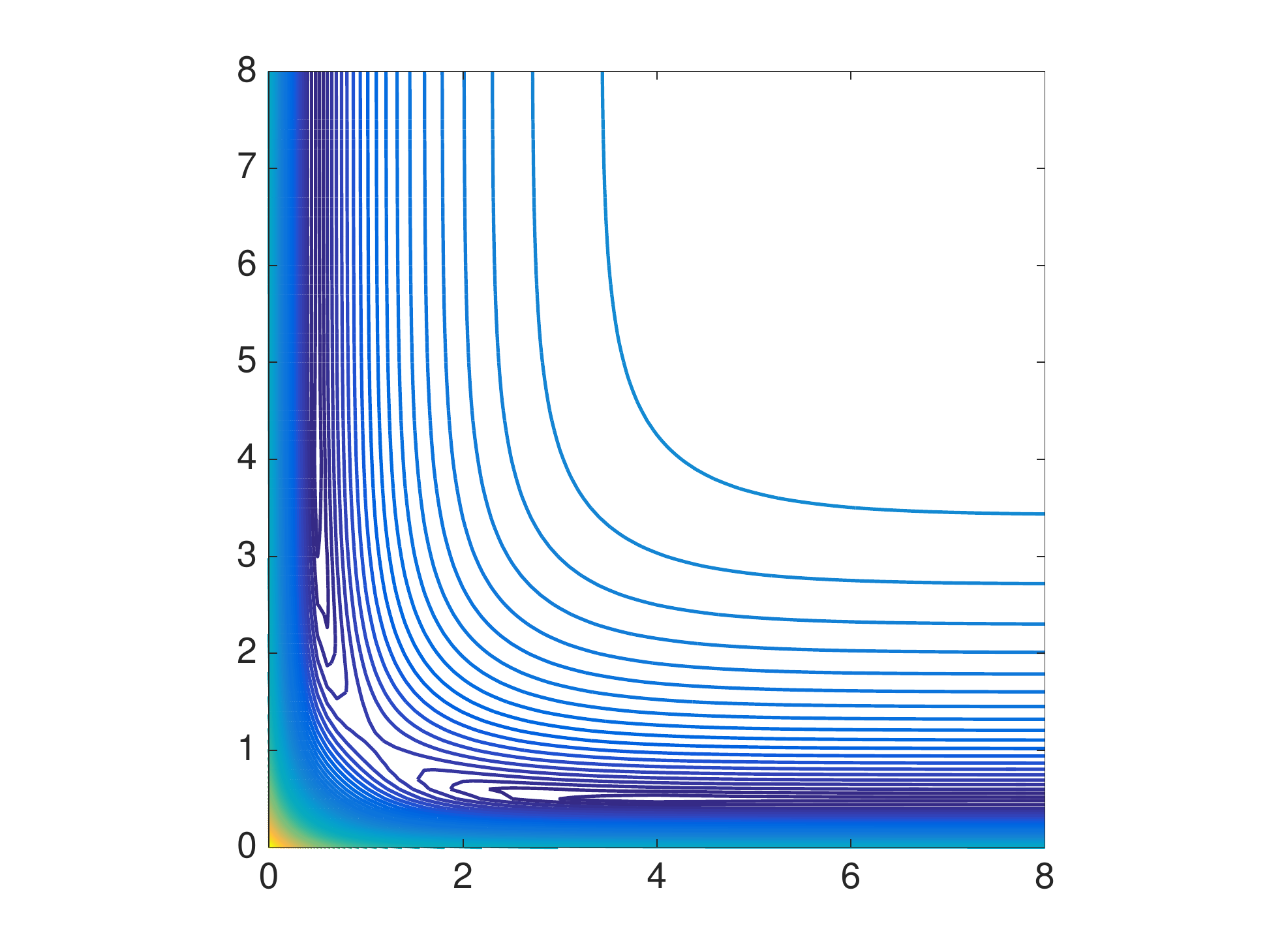}}; 
 \end{tikzpicture}

\caption{Sloppiness for different choices of model prediction map for the sum of exponentials (Example \ref{eg:SumExpInfDelta}). For a given parameter $p_0 = (4, 1/2)$, we draw the level curves of $\sqrt{d(\cdot,(4, 1/2))}$ for timepoints $\{1/3,1,3\}$ on the left, for timepoints $\{1/9,1/3\}$ in the center, and for timepoints $\{1,3\}$ on the right are shown.  Taking the square root changes the spacing of the level curves, but not on their shape.
 }
\label{fig:VaryingTimepoints}
\end{figure}

 \begin{figure}[h!]

\begin{tikzpicture}
\node    at  (-4.1,0) {\includegraphics[trim = 20mm 5mm 19mm 5mm, clip, width=4.5cm]{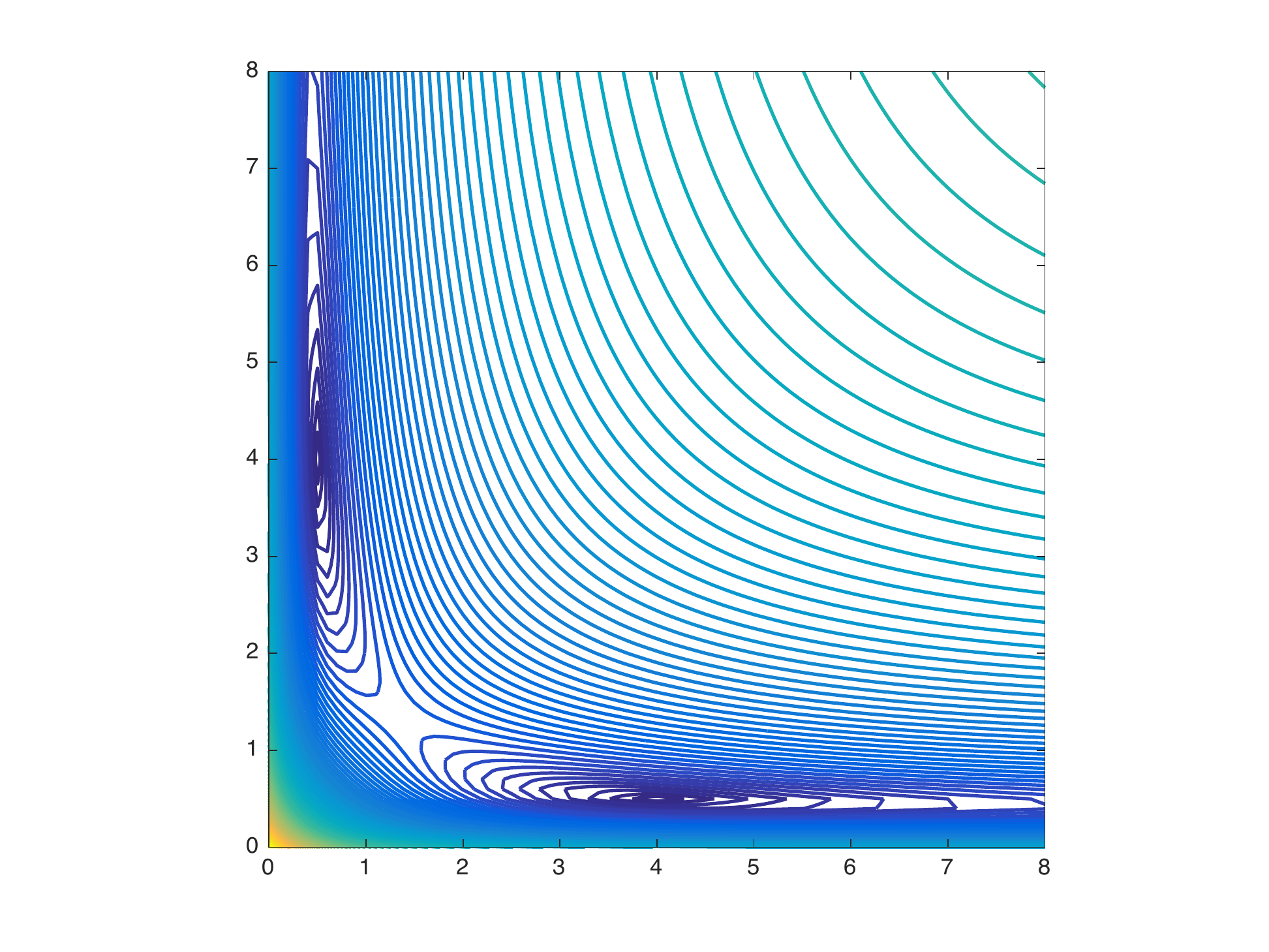}};

\node at (0,0) {\includegraphics[trim = 20mm 5mm 19mm 5mm, clip, width=4.5cm]{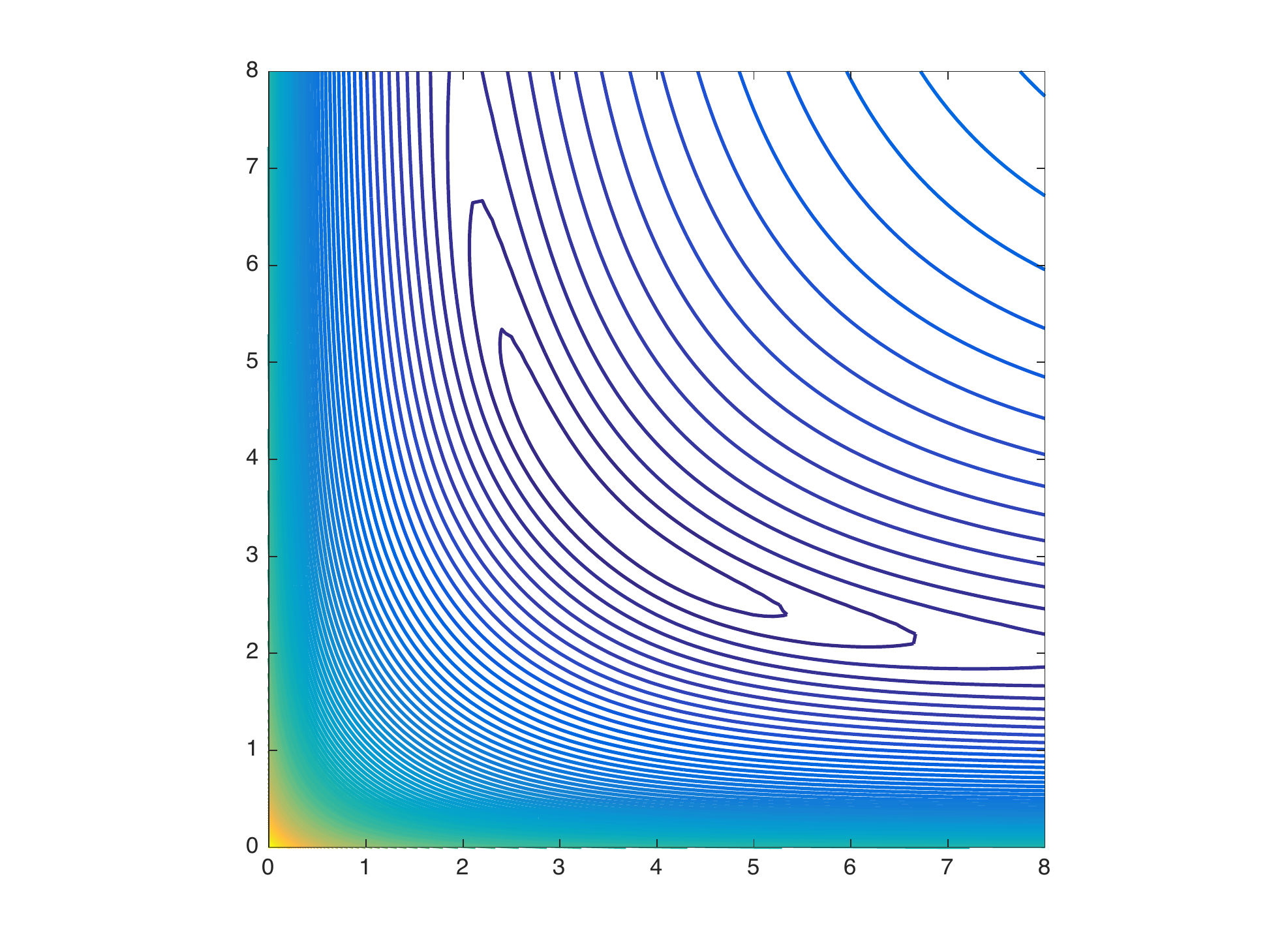}};

\node at (4.1,0) {\includegraphics[trim = 20mm 5mm 19mm 5mm, clip,  width=4.5cm]{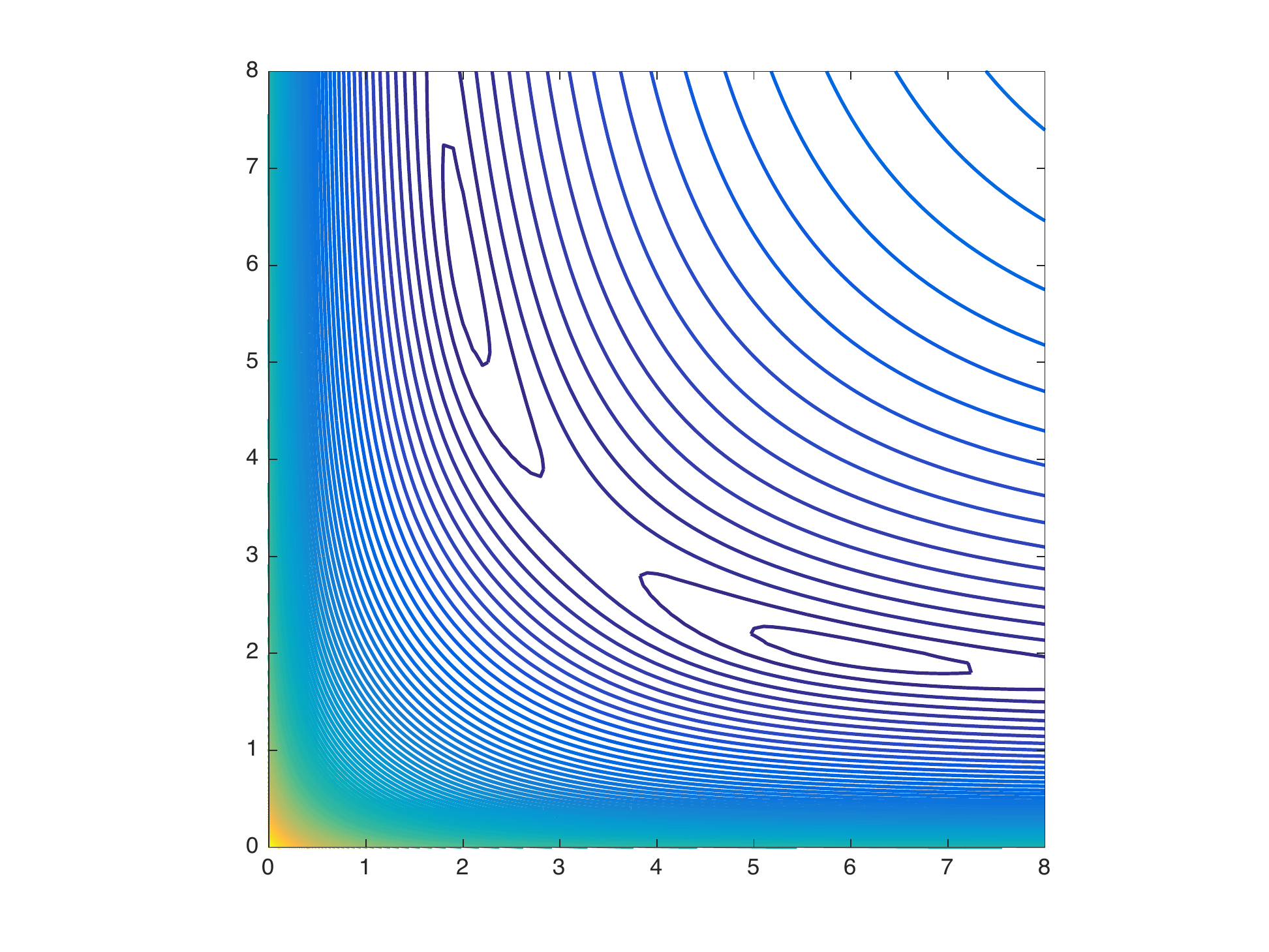}}; 
 \end{tikzpicture}

\caption{Sloppiness at different parameters given a choice of model prediction map for the sum of exponentials (Example \ref{eg:SumExpInfDelta}). With the model prediction map given by timepoints  $\{1/3,1,3\}$, we draw the level curves of $\sqrt{d(\cdot,(4, 1/2))}$  on the left, of $\sqrt{d(\cdot,(3, 3))}$ in the center, and of $\sqrt{d(\cdot,(6, 2))}$on the right are shown.  Taking the square root changes the spacing of the level curves, but not on their shape.
 }
\label{fig:VaryingParameters}
\end{figure}

\end{eg}

\subsection{Sloppiness and practical identifiability}

Determining the practical identifiability of a model corresponds to asking whether one can arrive to some estimate of the parameter from noisy data, that is, whether based on an assumption on measurement noise, noisy data constrains the parameter value to a bounded region of parameter space. Part of the literature uses the FIM in the manner of infinitesimal sloppiness to define practical identifiability(see for example \cite{Vajda1989,Chis2016:SloppinessIdentifiability}), but we will see in Example \ref{eg-NotSloppyNotPI} that this method of evaluating practical identifiability can lead to problems. We thus favor an approach more in line with Raue et al \cite{Raue2009}. 

Practical identifiability depends on the method used for parameter estimation. We focus on practical identifiability for maximum likelihood estimation, one of the most widely used methods for parameter estimation (see, for example \cite{ll:LjungSystemIdentification}). Accordingly, in the remaining of this section, we consider models $(M,\phi,\psi,d_P)$ with a choice of model prediction map, a specific assumption of the probability distribution of measurement noise and a choice of reference metric on $P$ such that maximum likelihood estimates exist for generic data.

For the noisy data point $z_0\in Z$, supposing the existence of a unique maximum likelihood estimate $\hat{p}(z_0)$ (i.e. supposing the model is generically identifiable, see Proposition \ref{prop-MLEequiv} below), we define an \emph{$\epsilon$-confidence region ${U}_\epsilon(z_0)$} as follows:
\begin{align*}
U_{\epsilon}(z_0) = \{p \in P \mid - \log \psi(p,z_0) < \epsilon \}.
\end{align*}
The $\epsilon$-confidence region therefore denotes the set of parameters that fit the data at least as well as some cutoff quality of fit, predicated on $\epsilon$. The set $U_{\epsilon}(z_0)$ is often known as a Likelihood-based confidence region \cite{Vajda1989, gc-rlb:statsBook}, and is intimately connected with the Likelihood Ratio Test: Suppose we had a null hypothesis $\mathcal{H}_0$ that data $z_0$ was generated (modulo noise) through a parameter $p_0$, and we wished to test the alternative hypothesis $\mathcal{H}_1$ that $z_0$ was generated through some other parameter. By definition, a Likelihood Ratio test would reject the null hypothesis when
\begin{align*}
\Lambda(p_0,z_0):= \frac{\psi(p_0,z_0)}{\psi(\hat{p}(z_0),z_0)} \leq k^*,
\end{align*}
where $k^*$ is a critical value, with the significance level $\alpha$ equal to the probability $\text{Pr}(\Lambda(z_0)\leq k^* | \mathcal{H}_0)$ of rejecting the null hypothesis when it is in fact true. The set of parameters such that the nul hypothesis is not rejected at significance level $\alpha$ is
\[ \{ p' \in P \mid \\log\psi(p',z_0)<-\log-\psi(\hat{p}(z_0),z_0)-\log k^*\},\]
that is, $U_\epsilon(z_0)$, where $\epsilon=-\log-\psi(\hat{p}(z_0),z_0)-\log k^*$.

\begin{prop}[{closely related to \cite[Theorem 2]{eac-bjtm:ParameterRedundency}}]\label{prop-MLEequiv}
Let $(M,\phi,\psi)$ be a mathematical model with a model prediction map, and a specific assumption on measurement noise. Suppose that maximum likelihood estimates exist for generic data. If $\phi$ and $\psi$ are real-analytic, then for almost all $z_0 \in Z$, the set of maximum likelihood estimates $\hat{p}(z)$, consists of exactly one equivalence class of $\sim\!\!_{M,\phi}$.
\end{prop}
\begin{proof}
Let $z_0\in Z$ be a generic data point. Solving the likelihood equation corresponds to finding the model prediction ``closest'' to the noisy data, as measured via the negative log-likelihood. We can assume without loss of generality that there is a unique solution to the likelihood equations. Indeed, under our assumptions, the set of data points where the closest model prediction is not unique will be contained in the zero set of analytic functions. Thus, the set of maximum likelihood estimates will consist of a single equivalence class. We can further assume that this equivalence class has generic size.
\end{proof}

\begin{rmk} 
The ML degree \cite{sh-ak-bs:sle}, where the acronym ``ML'' stand for maximum likelihood, is defined as the number of complex solutions to the likelihood equations (for generic data). The ML degree is an upper bound for the number of solutions for the maximum likelihood equation, in particular it is an upper bound on the size of the equivalence classes when maximum likelihood estimates exist.
\end{rmk}

Even if for generically identifiable models the maximum likelihood estimate is unique with probability one, the parameter may not be identifiable in practice, meaning that noisy data does not constrain the parameter value to a bounded region of parameter space for a significant portion of the data space. More precisely, we refine the definition of Raue et al \cite{Raue2009}:

\begin{defn}[Practical identifiability]
Let $(M,\phi,\psi,d_P)$ be a mathematical model with a model prediction map, a specific assumption on measurement noise and a choice of reference metric $d_P$ on $P$. Suppose that maximum likelihood estimates exist for generic data. Then $(M,\phi,\psi,d_P)$ is \emph{practically identifiable at significance level $\alpha$} if and only if for generic $z_0\in Z$, there is a unique maximum likelihood estimate and the confidence region $U_{\epsilon}(z_0)$ is bounded with respect to the reference metric $d_P$, where $\epsilon$ satisfies
\begin{align*}
p'\in U_{\epsilon}(z_0) \Leftrightarrow \text{Pr}\Big(-\log \psi(p',\hat{z})<\epsilon ~\Big{ | }~ \hat{z}\in Z \text{ is a corruption of }\phi(\hat{p}(z_0)\Big)=1-\alpha.
\end{align*}
The model $M$ is \emph{practically unidentifiable at significance level $\alpha$} if and only if there is a positive measure subset $Z'\subset Z$ such that for $z_0\in Z'$, the confidence interval $U_{\epsilon}(z_0)$ is unbounded with respect to the reference metric $d_P$ on $P$.
\end{defn}

A model is practically identifiable at significance level $\alpha$ if generic data imposes that the parameter estimate belongs to a bounded region of parameter space, but this confidence region could be very large. Hence practical identifiability in this sense may not necessarily be completely satisfactory to the practitioner. One can further quantify practical identifiability to take into account the size of confidence regions, see for example \cite{dvr-ja-ap:SloppinessDhruva}.

Sloppiness and practical identifiability are complementary concepts. Practically identifiable models can be very sloppy, for example if the estimation of one component of the parameter is much more precise than that of another, see example below. 

\begin{eg}[Practically identifiable, but sloppy]
Models with linear model prediction maps, Euclidean parameter space and standard additive Gaussian noise are always practically identifiable according to our definition, but these models can be arbitrarily sloppy. 

We consider a model with 2-dimensional Euclidean parameter space and a linear model prediction map $\phi$ given by $(a,b)\mapsto (10^Na,b)$. We assume further that the measurement noise is Gaussian with identity covariance matrix. By Proposition \ref{prop:LinearPredictions}, $d=d_{FIM,p_0}$ for any $p_0$ and at any scale, the level curves of $d$ are ellipses with aspect ratio $10^N$. 

Our assumption of additive Gaussian noise implies that for any $z_0$, for each $\epsilon>0$, the confidence interval $U_\epsilon(z_0)$ is an oval whose boundary ellipse is the level set of $d$ centered at the maximum likelihood estimate $\hat{p}(z_0)$. Thus the confidence intervals $U_\epsilon(z_0)$ are bounded for any $\epsilon>0$, and so the model is practically identifiable.
\done
\end{eg}
 
 In the following, we give an example of a model that is almost everywhere not sloppy at the infinitesimal scale, but is not practically identifiable. This model, however, exhibits some sloppiness at the non-infinitesimal scale. We see that the boundedness of level curves of $d$ almost every where does not imply the boundedness of confidence intervals almost everywhere.

 \begin{eg}[Not sloppy at the infinitesimal scale, but not practically identifiable]\label{eg-NotSloppyNotPI}
 Consider the mathematical model $(M,\phi,\mathcal{N}(\phi(p), I_2),d_2)$ given by 
 \begin{alignat*}{3}
 \phi :& [1/2,\infty)\times \RR & \to& \RR^2\\
         & (a,b)                          & \mapsto & \left(\frac{a}{a^2+b^2}, -\frac{b}{a^2+b^2}\right),
 \end{alignat*}
 additive Gaussian noise with identity covariance matrix, and parameter space $P=[1/2,\infty)\times \RR$ equipped with the usual Euclidean metric.
 
 The model prediction map $\phi$ is a conformal mapping  that maps the closed half plane $[1/2,\infty)\times \RR$ to the closed disc of radius 1 centered at $(1,0)$ minus the origin. Since it is a conformal mapping, it preserves angles, and so infinitesimal circles are sent to infinitesimal circles. Under our assumptions on measurement noise and with the standard Euclidean metric as a reference metric on $P$, the model is not sloppy at all  at the infinitesimal scale at parameters belonging to the open half plane $(1/2,\infty)\times \RR$, but becomes increasingly sloppy at larger and larger scale, especially away from the parameter $(1/2,0)$. The injectivity of the map $\phi$ on $P=[1/2,\infty)\times \RR$ implies that the model $(M,\phi)$ is globally identifiable. 
 
Our assumption on measurement noise implies that the maximum likelihood estimate is the parameter whose image is closest to the data point $z_0$, it will exist for any data point outside the closed half line $(-\infty,0]\times \{0\}$. The confidence region $U_\epsilon(z_0)$ is then the preimage of the Euclidean open disc of radius $\epsilon$ centered at $z_0$. Whenever the closure of this open disc contains the origin, the corresponding confidence region will be unbounded. 
 \done
 \end{eg}
 
The final example illustrates that the uniqueness of the maximum likelihood estimate is independent from the boundedness of the confidence regions:
 
\begin{eg}[Bounded confidence regions but not practically identifiable]
 Consider the mathematical model $(M,\phi,\mathcal{N}(\phi(p), I_2),d_2)$ with model prediction map
 \begin{alignat*}{3}
 \phi :& [1/2,\infty)\times \RR & \to& \RR\\
         & (a,b)                          & \mapsto & a^2+b^2,
 \end{alignat*}
 additive Gaussian noise with identity covariance matrix, and parameter space $P=[1/2,\infty)\times \RR$ equipped with the usual Euclidean metric. The equivalence class of the model-data equivalence relation are the concentric circles $\{(a,b)\mid a^2+b^2=r\}$ for $r\geq0$, and so the model is generically non-identifiable. By Proposition \ref{prop-MLEequiv}, the set of maximum likelihood estimates for a generic $(a_0,b_0)\in \RR$ is also a circle centered at the origin and the model is practically non-identifiable on any open neighborhood of $(a_0,b_0)$. On the other hand, as the measurement noise is assumed to be Gaussian and the equivalence classes of $\sim\!\!_{M,\phi}$ are bounded, any confidence region will be bounded as well. Indeed, the confidence region $U_{\epsilon}((a_0,b_0))$ will be either an open disk $\{(a,b)\in \RR^2 \mid a^2+b^2<a_0^1+b_0^2+\epsilon\}$, when $\epsilon>a_0^1+b_0^2$, or an open ring $\{(a,b)\in \RR^2 \mid a_0^1+b_0^2-\epsilon<a^2+b^2<a_0^1+b_0^2+\epsilon\}$, otherwise.\done
\end{eg}


\section{Future of sloppiness}

There are a number of interesting future directions for the theory and application of sloppiness. While we explained sloppiness via identifiability, this is only the beginning. An important next step is understanding sloppiness in the context of existing inference and uncertainty quantification theory. In terms of applications, there are some models where the reference metric on parameter space is non-Euclidean and we believe the computation of multiscale sloppiness can be adapted. While beyond the expertise of the authors, we would be excited to learn how the presented geometry of sloppiness extends to stochastic differential equations. 

We highlighted how sloppiness is a local property, dependent on the parameter and timepoints of experiment. This dependence is reflected in model selection studies where a different model is selected depending on the choice of timepoints  \cite{Silk:2014} or experimental stimulus dose \cite{Gross:2016}. We believe quantifying the shape of $\delta$-sloppiness in relation to identifiability will have direct impact on parameter estimation. 


\section{Acknowledgements}
HAH gratefully acknowledges the late Jaroslav Stark for posing this problem. The authors thank Carlos Am\'endola, Murad Banaji, Mariano Beguerisse D\'iaz, Sam Cohen, Ian Dryden, Paul Kirk, Terry Lyons, Chris Meyers, Jim Sethna, Eduardo Sontag, Bernd Sturmfels, and Jared Tanner for fruitful discussions. Additionally, we thank the anonymous referees for their helpful comments. This paper arises from research done while ED was a postdoctoral research assistant at the Mathematical Institute in Oxford funded by the John Fell Oxford University Press (OUP) Research Fund, and DVR was supported by the EPSRC Systems Biology Doctoral Training Center. ED is now supported by an Anne McLaren Fellowship from the University of Nottingham. HAH and DVR began discussions at the 2014 Workshop on Model Identification funded by KAUST KUK-C1-013-04.  HAH was supported by EPSRC Fellowship EP/K041096/1 and now a Royal Society University Research Fellowship.


\bibliographystyle{plainurl}
\bibliography{SloppinessReferences}

\appendix


\section{Table summary of main examples}

\noindent
\begin{tabular}{| p{4.2cm} | p{7.7cm} |}
\hline
Example                                &  Fitting points to a line \ref{eg-Line1} \ref{eg-Line2}\ref{eg-Line3}\ref{eg-Line4} \\
\hline
Type                                       & Explicit time-dependent model\\
\hline
Parameter space $P$            &  $(a_0,a_1)\in \RR^2$\\
\hline
Variable ($x\in X$)                  &$x(t)\in \RR$ for $t\in \RR_{\geq 0}$\\
\hline
Measurable output ($y\in Y$) &  $x(t)$\\
\hline
Perfect data                            &  $(x(t_1),\ldots,x(t_N))$ for some $t_1<\cdots<t_N\in\RR_{\geq 0}$\\
\hline
Noisy data                               & $(x(t_1)+\epsilon_1,\ldots,x(t_N)+\epsilon_N)$\\
\hline
\end{tabular}

~\\

\noindent
\begin{tabular}{| p{4.2cm} | p{7.7cm} |}

\hline
Example                                & Two biased coins \ref{eg-2BiasedCoins1}, \ref{eg-2BiasedCoins1} \\
\hline
Type                                       & Finite discrete statistical  model\\
\hline
Parameter space $P$            & $[0,1]^3$ \\
\hline
Variable ($x\in X$)                  &Outcome of 1 instance of the experiment     \\
\hline
Measurable output ($y\in Y$) & Record of 1 instance of the experiment  \\
\hline
Perfect data                            & Probability distribution for $p\in P$ \\
\hline
Noisy data                               & Record of $N$ instances of the experiment\\
\hline

\end{tabular}

~\\

\noindent
\begin{tabular}{| p{4.2cm} | p{7.7cm} |}

\hline
Example                                & Sum of exponentials \ref{eg-SumExponentials1}, \ref{eg-SumExponentials2}, \ref{eg-SumExponentials3}, \ref{eg:SumExpInfDelta}\\
\hline
Type                                       & Explicit time dependant model\\
\hline
Parameter space $P$            & $(a,b)\in \RR_{\geq 0}^2$ \\
\hline
Variable ($x\in X$)                  & $x(t)\in \RR_{>0}$ for $t\in\RR_{\geq0}$ \\
\hline
Measurable output ($y\in Y$) &  $x(t)\in \RR_{>0}$ for $t\in\RR_{\geq0}$\\
\hline
Perfect data                            &  $(x(t_1),\ldots,x(t_N))$ for some $t_1<\cdots<t_N\in\RR_{\geq 0}$\\
\hline
Noisy data                               & $(x(t_1)+\epsilon_1,\ldots,x(t_N)+\epsilon_N)$\\
\hline

\end{tabular}

~\\

\noindent
\begin{tabular}{| p{4.2cm} | p{7.7cm} |}

\hline
Example                                & An ODE model with an exact solution \ref{eg-ODESoln1}\\
\hline
Type                                       & Polynomial ODE model\\
\hline
Parameter space $P$            & $(p_1,p_2)\in \RR^2_{>0}$ \\
\hline
Variable ($x\in X$)                  &  $(x_1(t),x_2(t))\in \RR^2_{>0}$ for $t\in \RR_{\geq 0}$\\
\hline
Measurable output ($y\in Y$) & $(x_1(t), x_2(t))$ \\
\hline
Perfect data                            &  $(x_1(t_1), x_2(t_1),\ldots,_1(t_N), x_2(t_N) )$ for some $t_1<\cdots<t_N\in\RR_{\geq 0}$\\
\hline
Noisy data                               & $(x_1(t_1)+\epsilon_{1,1}, x_2(t_1)+\epsilon_{2,1},\ldots,_1(t_N)+\epsilon_{1,N}, x_2(t_N)+\epsilon_{2,N} )$ \\
\hline

\end{tabular}

~\\

\noindent
\begin{tabular}{| p{4.2cm} | p{7.7cm} |}

\hline
Example                                & A non-linear ODE model \ref{eg-NonlinODE}\\
\hline
Type                                       & Polynomial ODE model \\
\hline
Parameter space $P$            &  $(p_1,p_2,p_3,p_4,p_5)\in P\subseteq \RR^5$\\
\hline
Variable ($x\in X$)                  &  $(x_1(t),x_2(t))\in \RR^2_{>0}$ for $t\in \RR_{\geq 0}$\\
\hline
Measurable output ($y\in Y$) &  $(x_1(t),x_2(t))\in \RR^2_{>0}$ for $t\in \RR_{\geq 0}$\\
\hline
Perfect data                            & \textbullet~ $(x_1(t_1), x_2(t_1),\ldots,_1(t_N), x_2(t_N) )$ for some $t_1<\cdots<t_N\in\RR_{\geq 0}$\\
                                                & \textbullet~  An exhaustive summary or input-output equations\\
\hline
Noisy data                               & $(x_1(t_1)+\epsilon_{1,1}, x_2(t_1)+\epsilon_{2,1},\ldots,_1(t_N)+\epsilon_{1,N}, x_2(t_N)+\epsilon_{2,N} )$\\
\hline

\end{tabular}

~\\

\noindent
\begin{tabular}{|p{4.2cm}|p{7.7cm}|}

\hline
Example                                & Gaussian mixtures \ref{eg-GaussianMixtures1}\\
\hline
Type                                       & Continuous parametric statistical model \\
\hline
Parameter space $P$            & $(\lambda, \mu,\sigma,\nu,\tau)\in [0,1]\times \RR\times \RR_{\geq 0}\times \RR\times \RR_{\geq 0}$ \\
\hline
Variable ($x\in X$)                  &  a characteristic $x\in \RR_{\geq 0}$ of a mixed polulation\\
\hline
Measurable output ($y\in Y$) & $x\in \RR_{\geq 0}$ \\
\hline
Perfect data                            &  \textbullet~  Probability distribution of  $x$ for some $(\lambda, \mu,\sigma,\nu,\tau)$\\
                                               
                                                  & \textbullet~ Value of the cdf (or the pdf) at general $x_1,\ldots,x_11\in\RR$\\
                                                 
                                                  & \textbullet~ Value of the moment generating function at general  $t_1,\ldots,t_11\in(-a,a)$\\
                                           
                                                 & \textbullet~ 11 generic moments (or the first 7)\\
                                                 \hline

Noisy data                               &\textbullet~ 
                                                  Measurements from a finite sample with some measurement error\\
                                               
                                                 & \textbullet~  Empirical distribution function from a finite sample (or its numerical derivative)\\
                                                  & \textbullet~  11 generic sample moments from a finite sample (or the first 7)\\
                                          
\hline     
\end{tabular}

~\\

\noindent
\begin{tabular}{| p{4.2cm} | p{7.7cm} |}

\hline
Example                                & Linear parameter-varying model \ref{eg-linparvar} \\
\hline
Type                                       & ODE model \\
\hline
Parameter space $P$            &  $p\in \RR^{r-1}$\\
\hline
Variable ($x\in X$)                  &  $x(t)\in\RR^m$ for $t\in \RR_{\geq 0}$ \\
\hline
Measurable output ($y\in Y$) &  $y=Cx(t)$ for $t\in \RR_{\geq 0}$\\
\hline
Perfect data                            &  $(y(t_1),\ldots,y(t_N))$ for some $t_1<\cdots<t_N\in\RR_{\geq 0}$ \\
\hline
Noisy data                               & $(y(t_1)+\epsilon_1,\ldots,y(t_N)+\epsilon_N)$ \\
\hline

\end{tabular}

\vspace{1cm}


\end{document}